\documentclass[12pt, twoside, leqno]{amsart}
\usepackage{amsmath,amsthm}
\usepackage{amssymb,latexsym, amscd}
\usepackage{enumerate}
\usepackage{color}
\usepackage{bm}
\usepackage[all]{xy}
\usepackage{amscd}

\usepackage[top=30truemm,bottom=30truemm,left=25truemm,right=25truemm]{geometry}
\usepackage{mathtools}
\mathtoolsset{showonlyrefs=true}

\usepackage{tikz}
\usetikzlibrary{cd}

\makeatletter
\@namedef{subjclassname@2020}{\textup{2020} Mathematics Subject Classification}
\makeatother

\newcommand{\Z}{\mathbb{Z}}
\newcommand{\N}{\mathbb{N}}
\newcommand{\bN}{\mathbb{N}}
\newcommand{\Q}{\mathbb{Q}}
\newcommand{\R}{\mathbb{R}}

\newcommand{\Zp}{\mathbb{Z}_p}

\newcommand{\cC}{\mathcal{C}}

\newcommand{\cL}{\mathcal{L}}
\newcommand{\cM}{\mathcal{M}}
\newcommand{\cO}{\mathcal{O}}
\newcommand{\cP}{\mathcal{P}}
\newcommand{\cS}{\mathcal{S}}
\newcommand{\cT}{\mathcal{T}}

\newcommand{\cZ}{\mathcal{Z}}

\newcommand{\Cmodsim}{\cC/\text{\lower.5ex\hbox{$\sim$}}\,\, }
\newcommand{\modsim}{\text{\lower.5ex\hbox{$\sim$}}\,\, }

\newcommand{\fP}{\mathfrak{P}}
\newcommand{\fp}{\mathfrak{p}}

\newcommand{\wtil}[1]{\widetilde{#1}}
\newcommand{\ol}[1]{\overline{#1}}

\DeclareMathOperator{\Gal}{Gal}

\DeclareMathOperator{\Ker}{Ker}

\DeclareMathOperator{\Fitt}{Fitt}

\DeclareMathOperator{\Cl}{Cl}
\DeclareMathOperator{\ord}{ord}
\DeclareMathOperator{\Lat}{Lat}

\DeclareMathOperator{\Hom}{Hom}
\DeclareMathOperator{\Ext}{Ext}

\DeclareMathOperator{\aug}{aug}
\DeclareMathOperator{\pe}{pe}
\DeclareMathOperator{\adm}{adm}
\DeclareMathOperator{\rea}{real}
\DeclareMathOperator{\prim}{prime}
\DeclareMathOperator{\ram}{ram}

\theoremstyle{plain}
\newtheorem{thm}{Theorem}[section]
\newtheorem{lem}[thm]{Lemma}

\newtheorem{prop}[thm]{Proposition}
\newtheorem{cor}[thm]{Corollary}
\newtheorem{claim}[thm]{Claim}

\theoremstyle{definition}
\newtheorem{defn}[thm]{Definition}

\newtheorem{rem}[thm]{Remark}
\newtheorem{eg}[thm]{Example}

\title{On the Galois module structure of minus class groups}

 \author[C.~Greither]{Cornelius Greither}
 \address{Fakult\"at Informatik, Universit\"at der Bundeswehr M\"unchen,
 85577 Neubiberg, Germany}
 \email{cornelius.greither@unibw.de}

 \author[T.~Kataoka]{Takenori Kataoka}
\address{Department of Mathematics, Faculty of Science Division II, Tokyo University of Science.
1-3 Kagurazaka, Shinjuku-ku, Tokyo 162-8601, Japan}
\email{tkataoka@rs.tus.ac.jp}


\begin{document}


 \begin{abstract} 
The main object of this paper is the minus class groups associated to CM-fields as Galois modules.
In a previous article of the authors, we introduced a notion of equivalence for modules and determined the equivalence classes of the minus class groups.
In this paper, we show a concrete application of this result.
We also study how large a proportion of equivalence classes can be realized as classes of minus class groups.
 \end{abstract}
	


\maketitle
   
\section{Introduction}

In some sense this paper is a continuation of a paper of the authors \cite{GK}.
The setting is given by a CM-field $L$ which
is an abelian extension of a totally real field $K$, the Galois
group being called $G$.
Let $\Z[G]^- := \Z[1/2][G]/(1 + j)$ be the minus part of the group ring, where $j$ denotes complex conjugation.
For any $\Z[G]$-module $M$, the minus part $M^-$ is defined 
as $M^- := \Z[G]^- \otimes_{\Z[G]} M$.
We consider the minus part $\Cl_L^{T, -}$ of the $T$-modified class group $\Cl_L^T$ (the technicalities will be
explained in \S \ref{ss:clgp}).

The question is, as in previous work:
How much can we say about the structure of $\Cl_L^{T, -}$
as a module over $\Z[G]^-$, based on equivariant $L$-values
and field-theoretic information (ramification and the like)
attached to $L/K$?  The Fitting ideal of the Pontryagin dual $\Cl_L^{T, -, \vee}$ had been determined,
more and more generally and unconditionally, by the first author \cite{Gr}, Kurihara \cite{Ku}, 
and Dasgupta--Kakde \cite{DK}. Somewhat unexpectedly the 
Fitting ideal of the non-dualized module $\Cl_L^{T, -}$
was determined later, in the paper
\cite{AK} by Atsuta and the second author.
In particular, as a consequence of lengthy computations, 
we obtained an inclusion relation
\[
\Fitt(\Cl_L^{T, -}) \subset \Fitt(\Cl_L^{T, -, \vee}).
\]

In \cite{GK}, we introduced a new notion of equivalence for the category $\cC$ of finite $\Z[G]^-$-modules, denoted simply by $\sim$.
Moreover, we described the equivalence class of $\Cl_L^{T, -}$ (see \S \ref{sec:1}).
A guiding principle in defining $\sim$ is to regard $G$-cohomologically trivial modules as zero.
From the view point of Fitting ideals, this implies that $\sim$ ignores the contribution of invertible ideals.
Therefore, the analytic factor coming from $L$-functions,
which was present in the earlier descriptions of the Fitting ideal,
is lost.
Nevertheless, the notion $\sim$ is useful enough.
For instance, as a first application of the theory, we have reproved the aforementioned inclusion relation.

This paper deals with two problems concerning the structure of $\Cl_L^{T, -}$ from the viewpoint of $\sim$.
The first one is to obtain a more concrete application of the notion of $\sim$.
The second is to obtain some idea how large a proportion of equivalence classes is realized 
via classes of class groups. 
It turns out that this proportion tends to be remarkably small.
In \S \ref{ss:1-1} and \S \ref{ss:1-2} respectively, we will briefly explain these results.

\subsection{A concrete application of the equivalence}\label{ss:1-1}

It is our purpose here to show that our description of $\Cl_L^{T, -}$ up to $\sim$ can actually
lead to concrete predictions concerning the structure of $\Cl_L^{T, -}$.
Of course our predictions will fall short of determining
the isomorphism class a priori -- that would be way too
ambitious.

Let $p \geq 3$ be a prime number.
We consider the case where $L^+/K$ is a $p$-extension, where as usual $L^+$ denotes the maximal totally real subfield of $L$.
In this case, $L/K$ has a unique intermediate field $F$ such that $F/K$ is a quadratic extension.
Then $F$ must be a CM-field satisfying $F^+ = K$.

The following is the main result here.
Let $\ord_p(-)$ be the additive $p$-adic valuation normalized by $\ord_p(p) = 1$.
In Lemma \ref{lem:no_root}, we will see that the $T$-modification is unnecessary because of the second assumption.

\begin{thm}\label{thm:main1}
Suppose the following:
\begin{itemize}
\item
$L^+/K$ is a cyclic $p$-extension of degree $p^r$ for some $r \geq 1$.
\item
$L$ has no non-trivial $p$-th roots of unity.
\item
There is a unique prime $v$ of $K$ that is ramified in $L^+/K$ and split in $L/L^+$.
\item
$v$ is totally ramified in $L^+/K$.
\item
$\ord_p(\# \Cl_F^-) = 0$.
\end{itemize}
Then $\ord_p(\# \Cl_L^-)$ is in the set
\[
\{r, 2r, 3r, \dots, pr \} \cup \{pr+1, pr+2, pr+3, \dots\}.
\]
In other words, $\ord_p(\# \Cl_L^-)$ is nonzero and either divisible by $r$ or larger than $pr$.
\end{thm}

The proof will be given in \S \ref{sec:4}.
In \S \ref{ss:app_ex}, we will also check these predictions on some explicit examples.
The results indicate that the theorem may be sharp.

\subsection{Realization problem}\label{ss:1-2}

The second topic in this paper is the question: Which finite $\Z[G]^-$-modules 
can be equivalent to a class group $\Cl_L^{T, -}$ a priori?

To be more precise, let us fix an abstract finite abelian group $\Gamma$.
For simplicity, we assume that the order of $\Gamma$ is odd.
Let $\cC$ be the category of finite $\Z'[\Gamma]$-modules, where we put $\Z' = \Z[1/2]$.
In this setting we have the notion of equivalence $\sim$ on $\cC$.
Let $\Cmodsim$ be the set of equivalence classes.
It is known that $\Cmodsim$ can be regarded as a commutative monoid with 
respect to direct sums.

Let us consider various abelian CM-extensions $L/K$ 
such that $\Gal(L^+/K) \simeq \Gamma$.
Since the order of $\Gamma$ is odd, as in \S \ref{ss:1-1}, there is a unique intermediate CM-field $F$ satisfying $F^+ = K$ and $\Gal(L/F) \simeq \Gamma$.
Then we have an identification
\[
\Z[\Gal(L/K)]^- \simeq \Z'[\Gamma]
\]
induced by the inclusion $\Gamma \subset \Gal(L/K)$.
Therefore, we may talk about the class of the minus class group $\Cl_L^{T, -}$ in $\Cmodsim$.

We call an element of $\Cmodsim$ a {\it realizable} class 
if it is the class of $\Cl_L^{T, -}$ for some extension $L/K$
described above (both $L$ and $K$ vary).
Let $\cZ^{\rea} \subset \Cmodsim$ be the subset of realizable classes.
Now the question is to study the size of $\cZ^{\rea}$.

Our main results involve another subset $\cZ^{\adm}$ of $\Cmodsim$, whose elements we call {\it admissible} classes.
The definition of $\cZ^{\adm}$ will be given in \S \ref{ss:formulation}.
Here we only mention that $\cZ^{\adm}$ is defined in an algebraic way (independent from arithmetic) so that we have $\cZ^{\rea} \subset \cZ^{\adm}$.
Also, $\cZ^{\adm}$ is by definition a submonoid, while it is not clear whether so is $\cZ^{\rea}$ a priori.

Now the problem splits naturally into two sub-problems:
\begin{itemize}
\item[(a)]
Do we have $\cZ^{\rea} = \cZ^{\adm}$?
\item[(b)]
What is the monoid structure of $\cZ^{\adm}$?
\end{itemize}
Note that (a) is an arithmetic problem; to prove $\cZ^{\rea} = \cZ^{\adm}$, we have to construct suitable extensions $L/K$.
On the other hand, (b) is an algebraic problem.

As for (a), we will give the following affirmative answer, which will be proved in \S \ref{sec:3}:

\begin{thm}\label{thm:real_p}
For any finite abelian group $\Gamma$ whose order is odd,
we have $\cZ^{\rea} = \cZ^{\adm}$.
\end{thm}

In particular, $\cZ^{\rea}$ is a submonoid of $\Cmodsim$.
Note that we will moreover have a concrete condition on the base field $K$ to realize each admissible class.
In particular, if $\Gamma$ is a cyclic group, every admissible class is realized by $K = \Q$ (see Remark \ref{rem:K_recipe}).

As for (b), we need to introduce a finite set $\cT$, which depends only on the group structure of $\Gamma$.
When $\Gamma$ is a $p$-group, the set $\cT$ is identified with the set of pairs $(I, D)$ such that 
\begin{itemize}
\item
$I \subset D \subset \Gamma$ are subgroups,
\item
$I$ is non-trivial, and 
\item
$D/I$ is cyclic.
\end{itemize}
The general definition will be given in Definition \ref{defn:sets}.
The main result for (b) is the following:

\begin{thm}\label{thm:main}
The following hold:
\begin{itemize}
\item[(1)]
Suppose $\Gamma$ is cyclic or is a $p$-group for some prime number $p$.
Then $\cZ^{\adm}$ is a free monoid of rank $\# \cT$.
\item[(2)]
Otherwise, $\cZ^{\adm}$ is not a free monoid.
\end{itemize}
\end{thm}

Let us focus on $p$-groups.
As an immediate corollary of Theorems \ref{thm:real_p} and \ref{thm:main}, we obtain the following:

\begin{cor}\label{cor:main2}
Suppose that $\Gamma$ is a $p$-group for some prime number $p$.
Then the subset $\cZ^{\rea}$ of $\Cmodsim$ is a commutative monoid 
that is free of rank $\# \cT$.
\end{cor}

Here is a brief discussion on the relative size of $\cZ^{\rea}$ within $\Cmodsim$
when $\Gamma$ is a non-trivial $p$-group for some prime number $p$.
Note that the structure of $\Cmodsim$ is already discussed in \cite{GK}.

\begin{itemize}
\item
If $\Gamma$ is of order $p$,
then both $\cZ^{\rea}$ and $\Cmodsim$ are free of rank one.
Indeed, we may even prove $\cZ^{\rea} = \Cmodsim$, that is, all finite $\Gamma$-modules up to equivalence occur as minus class groups (see Theorem \ref{thm:p_real}).
\item
If $\Gamma$ is cyclic  of order $p^2$, in \cite{GK} we have shown that $\Cmodsim$ is not a free monoid and the rank of the abelian group associated to $\Cmodsim$ is $4p-2$.
This results from the Heller--Reiner classification on the $\Gamma$-lattices, given in \cite{CR}.
On the other hand, it is easy to see that $\# \cT = 3$.
Therefore, $\cZ^{\rea}$ is much smaller than $\Cmodsim$ just as $3$ is smaller than $4p-2$.
\item
For any other $\Gamma$, the classification of $\Gamma$-lattices is a deep result (see Heller and Reiner \cite{HR1}, \cite{HR2}).
In particular, it is known that
the rank of the abelian group associated to $\Cmodsim$ is infinite.
On the other hand, $\# \cT$ is of course always finite.
Therefore, $\cZ^{\rea}$ is again much smaller than $\Cmodsim$.
\end{itemize}

So only a small portion of equivalence classes are realized as minus class groups.
It is interesting to observe that the situation 
around plus components is totally in contrast (see Remark \ref{rem:plus}).

\subsection{Organization of this paper}

We begin by reviewing results from \cite{GK} in \S \ref{sec:1}.
We obtain a description of the equivalence classes of minus class groups, and then introduce the notion of realizable and admissible classes.

In \S \ref{sec:4}, we prove Theorem \ref{thm:main1} and check numerical examples.
In \S \ref{sec:3}, we prove Theorem \ref{thm:real_p}.
In \S \S \ref{sec:6}--\ref{sec:7}, we prove Theorem \ref{thm:main}.

\section{Review of the equivalence relation}\label{sec:1}

We begin with a review of the notion of equivalence introduced 
by the authors \cite{GK}.

In \S\S \ref{ss:equiv}--\ref{ss:shift}, we recall the equivalence relation $\sim$, the re-interpretation of $\sim$ in terms of lattices, and the notion of shift.
The theory of shifts is basically known from 
work of the second author \cite{Ka}, but we
add a new aspect, linking it with Heller's loop operator for lattices.

After fixing the arithmetic setup in \S \ref{ss:clgp}, we obtain the description of the equivalence class of minus class groups in \S \ref{ss:eq_clgp}.
In \S \ref{ss:formulation}, we introduce the notion of realizable classes and admissible classes.

\subsection{The equivalence relation}\label{ss:equiv}

Let $R$ be a commutative ring that is Gorenstein of Krull dimension one.
Typical examples include finite group rings such as $\Z[G]^-$ and $\Z'[\Gamma]$, where
\begin{itemize}
\item
$G$ is a finite abelian group and $(-)^-$ is considered 
with respect to a fixed element $j \in G$ whose exact order is $2$.
\item
$\Gamma$ is a finite abelian group.
\end{itemize}
Note that in \cite{GK} we established the general theory for Gorenstein 
rings of finite Krull dimension, but in this paper we only need dimension one cases.

Let $\mathcal C$ be the category of finitely generated torsion $R$-modules.
Let us write $\cP$ for the subcategory of $\cC$ that consists of 
modules whose projective dimensions over $R$ are at most one.
Note that when $R = \Z[G]^-$ or $R = \Z'[\Gamma]$ as above, $\cC$ consists of all finite $R$-modules and $\cP$ consists of all finite $R$-modules that are $G$-c.t.~or $\Gamma$-c.t., respectively (``c.t.'' is an abbreviation of ``cohomologically trivial'').

\begin{defn}
We define a relation $\sim$ on $\mathcal C$ 
as follows:
\begin{itemize}
\item[(a)] A {\it sandwich\/} is a module $M$ in $\cC$ with
a three-step filtration by submodules $0 \subset M' \subset M'' \subset M$
satisfying the following conditions:
\begin{itemize}
\item
The top quotient $M/M''$ and the
bottom quotient $M'/0=M'$ are both in $\cP$.
\item
The middle filtration quotient $M''/M'$ is in $\cC$.
\end{itemize}
The middle filtration quotient $M''/M'$
is called the {\it filling\/} of the sandwich. 

\item[(b)] Two modules $X$ and $Y$ in $\cC$ are 
{\it equivalent ($X\sim Y$),} if $X$ is the filling
of some sandwich $M$, $Y$ is the filling of some other sandwich $N$, and
$M$ and $N$ are isomorphic as $R$-modules. The isomorphism between 
$M$ and $N$ is not assumed to respect the filtrations.
\end{itemize}
\end{defn}

For basic properties of this notion and
in particular the nontrivial fact that
this is an equivalence relation we refer to \cite[Remark 2.4, Propositions 2.5--2.6]{GK}.
For now let us just
remark that one easily shows: $M \sim 0$ is equivalent to $M \in \cP$.
Indeed, the definition of $\sim$ arose from the idea of forcing that property.

The set of equivalence classes $\Cmodsim$ is equipped with a commutative monoid structure with respect to direct sums.
In the following, $\Cmodsim$ will always be studied as a commutative monoid.
For each module $M$ in $\cC$, we write $[M]$ for the equivalence class of $M$ in $\Cmodsim$.

\subsection{The equivalence classes via lattices}\label{ss:eq_lat}

In \cite[\S 4]{GK}, we established an interpretation of the equivalence relation $\sim$ via lattices.
Let us briefly review the results here.

Let $\Lat^{\pe}$ denote the commutative monoid of $R$-lattices up to projective equivalence.
Here, an $R$-lattice (which we sometimes simply call a lattice) is by definition a finitely generated torsion-free $R$-module.
Two $R$-lattices are projectively equivalent if they become isomorphic
after adding finitely generated projective $R$-modules.
We write
\[
\cL \sim_{\pe} \cL'
\]
if $\cL$ and $\cL'$ are projectively equivalent lattices.
The monoid structure of $\Lat^{\pe}$ is defined by direct sums.

\begin{defn}
We define an injective monoid homomorphism
\[
\Phi: \mathcal \Cmodsim \hookrightarrow \Lat^{\pe}
\]
as follows:
To each $X\in \mathcal C$, choose an epimorphism $F \to X$ from a finitely generated projective $R$-module $F$ to $X$ and define a lattice $\cL_X$ as its kernel.
Though $\cL_X$ depends on the choice of the epimorphism, 
it is proved in \cite[Theorem 4.2]{GK} that this induces
a well-defined map $\Phi$ that sends the class of $X$ to the class of $\cL_X$,
and that moreover
$\Phi$ is injective. 
The image of $\Phi$ is also discussed in \cite[Lemma 4.3]{GK}.
\end{defn}

\subsection{The shift operator}\label{ss:shift}

There are shift operators $\omega^n$ on the monoid $\Cmodsim$, 
see \cite[Definition 3.4]{GK} and following discussion.
($\omega^n(X)$ in this paper means $X[n]$ in \cite{GK}.)
Let us briefly define them.
\begin{defn}
For any $X\in \mathcal C$, taking a short exact sequence
\[
0 \to Y \to P \to X \to 0,
\]
where $P \in \cP$, we define $\omega^1(X) \sim Y$.
This $\omega^1$ is indeed well-defined and an automorphism.
We then define the general $\omega^n$ inductively by $\omega^{n+1} = 
\omega^1 \circ \omega^n$ for any integer $n$.
\end{defn}

On the lattice side, we have the Heller operator $\Omega$, which is an automorphism of $\Lat^{\pe}$. 
It works as follows: given a lattice $\mathcal L$,
take an exact sequence $0 \to \mathcal M \to F \to \mathcal L \to 0$,
again with $F$ finitely generated projective over $R$. 
Then $\Omega$ sends the class of $\mathcal{L}$ to the class of $\mathcal{M}$.

We can link our shift operator $\omega^1$ and the Heller operator $\Omega$.
The relation is as simple as possible.

\begin{lem} \label{commsq}
 There is a commutative square as follows:
\[
 \xymatrix{
  \Cmodsim \ar[r]^{\omega^1}_{\simeq} \ar@{^(->}[d]^{\Phi} &  \Cmodsim \ar@{^(->}[d]^\Phi \\
	\Lat^{\pe}  \ar[r]^{\Omega}_{\simeq} & \Lat^{\pe}. \\
  }
\]
\end{lem}

\begin{proof}
Take $X \in \cC$ and an exact sequence $0 \to Y \to P \to
X \to 0$ with $P \in \cP$, so $Y \sim \omega^1(X)$. 
Take compatible projective resolutions
so as to obtain a diagram
\[
\xymatrix{
  0 \ar[r] & \mathcal L' \ar[r] \ar@{^(->}[d] &  \mathcal L \ar[r] \ar@{^(->}[d] 
	       & \mathcal L'' \ar[r] \ar@{^(->}[d]  & 0 \\
	0 \ar[r] & F' \ar[r] \ar@{->>}[d] &  F \ar[r] \ar@{->>}[d]  & F'' \ar[r] \ar@{->>}[d]  & 0 \\
  0 \ar[r] & Y \ar[r]  &  P \ar[r]  & X \ar[r]  & 0, \\
}
\]
in which the modules $F',F,F''$
are finitely generated projective and the upper (lower) vertical arrows are injections 
(surjections respectively). Then $\mathcal L'$ represents $\Phi(Y)$ and
$\mathcal L''$ represents $\Phi(X)$.
Moreover, $\cL$ is projective over $R$ since $P \in \cP$ and $F$ is projective.
Therefore
we also have that $\mathcal L'$ represents $\Omega(\mathcal L'')$.
These observations imply that both $\Phi(\omega^1(X))$ and $\Omega(\Phi(X))$ are represented by $\cL'$, so the lemma follows.
\end{proof}

\subsection{The arithmetic setup}\label{ss:clgp}

We review the setting, which will be in force throughout the paper.
Assume that $L$ is a CM-field which
is an abelian extension of a totally real field $K$.
Write $G$ for $\Gal(L/K)$. Note that $G$ must have
even order;  it contains a privileged element
$j$ of order 2 given by complex conjugation.

We also have to discuss $T$-modification, in order to make the results
from \cite{AK} applicable.
Let $T$ be any finite set of prime ideals of
$K$ not containing any ramified prime. 
Let $T_L$ denote the set of primes of $L$ that lie above primes in $T$.

An ideal $J$ of $L$ is called
$T$-principal, if it admits a generator $x\in L^\times_T$, where
the latter group is defined as
\[
L^\times_T = \{x\in L^\times \mid  \ord_{\fP}(x-1)>0, \ \forall \fP \in T_L\}.
\]
The $T$-modified class group $\Cl_L^T$ is then defined as the group of all
fractional ideals of $L$ having support disjoint from $T$, modulo
the subgroup of $T$-principal ideals. This is a slight enlargement
of the usual class group $\Cl_L$. More precisely, there is a canonical
surjection $\Cl^T_L \to \Cl_L$, and its kernel
is an epimorphic image of $\bigoplus_{\fP \in T_L} \kappa(\fP)^\times$, where $\kappa(\fP)$ denotes the residue field at $\fP$.

The requirement for $T$ in \cite{AK} and many other papers is the following:
In addition to the conditions already stated, $T$ must be such that
$L^{\times}_T$ is $\Z$-torsion-free.
In  other words, we must have $\mu(L) \cap L^{\times}_T = \{1\}$, where $\mu(L)$ denotes the group of roots of unity in $L$.
Trivially this implies that $T$ cannot be empty. 
However, the following lemma implies that in certain cases this modification does not matter.
Let ${}_p \Cl_L$ and ${}_p \Cl_L^T$ be the $p$-Sylow subgroups of $\Cl_L$ and $\Cl_L^T$, respectively.

\begin{lem}\label{lem:no_root}
If $L$ has no non-trivial $p$-th roots of unity, there is a legitimate choice of $T$ such that ${}_p \Cl^T_L \simeq {}_p \Cl_L$.
\end{lem}

\begin{proof}
Let $f$ be the order of $\mu(L)$, which is prime to $p$ by the assumption.
By Tchebotarev's density theorem, one can find a prime $\fP$ of $L$ such that the order of $\kappa(\fP)^\times$ is
divisible by $f$ but not by $p$.
Set $T=\{\fp\}$ with $\fp=\fP \cap K$.
Then we have $\mu(L) \cap L^{\times}_T = \{1\}$ since $f \mid \#(\kappa(\fP)^\times)$.
On the other hand, the order of $\kappa(\fP)^{\times}$
 is prime to $p$, so in the $p$-part there is
no difference between $\Cl^T_L$ and $\Cl_L$.
\end{proof}

\subsection{The equivalence classes of class groups}\label{ss:eq_clgp}

When one takes \cite{AK} and \cite{GK} together,
one sees that using the notions of equivalence
and of shifting one can say a lot on $\Cl_L^{T, -}$.

 \let\phi=\varphi

We need a little more notation. Let $v$ run through the finite primes of $K$
ramifying in $L$. For each such $v$, let $I_v\subset G$ be the inertia group
and $\phi_v \in G/I_v$ the Frobenius at $v$. Define
\[
g_v = 1- \phi_v^{-1} + \# I_v \in \Z[G/I_v]; \quad  A_v = \Z[G/I_v]/(g_v),
\]
where $\# I_v$ denotes the order of $I_v$.

We will work over the ring $\Z[G]^-$ and define $\cC$ and $\cP$ accordingly.
Then \cite[Proposition 3.6]{AK} shows
the existence of a short exact sequence of $\Z[G]^-$-modules
\[
0 \to \Cl_L^{T, -} \to P \to \bigoplus_v A_v^- \to 0,
\]
 where $P$ is a $G$-c.t.~module (i.e., in $\cP$). Given the good
behaviour of shift under equivalence, this gives the following basic result:

\begin{thm}\label{thm:arith_main}
With all the notation introduced so far, we have
\[
\Cl_L^{T, -} \sim \bigoplus_v \omega^1 (A_v^-),
\]
where $v$ runs over the finite primes of $K$ that are ramified in $L$.
\end{thm}  \medskip

Note that the right hand side does not depend on the set $T$, and that the
variance under $T$ is hidden in the equivariant $L$-value, which is
not a part of the statement here. Nevertheless, to keep things technically
correct, one has to leave $T$ in at least formally.

\begin{lem}\label{lem:A_ct}
If $v$ is ramified or inert in $L/L^+$, then $A_v^-$ is in $\cP$.
\end{lem}

\begin{proof}
It is enough to show that $\Z_p \otimes A_v^- = \Z_p[G/I_v]^-/(g_v)$ is $G$-c.t.~ for any odd prime $p$.
First suppose that $v$ is ramified in $L/L^+$, that is, $j \in I_v$.
Then $j$ acts as $+1$ on $\Zp[G/I_v]$, so the minus part $\Zp[G/I_v]^-$ is already trivial.
Second suppose that $p \nmid \# I_v$.
In this case, $\Z_p[G/I_v]$ is $G$-c.t., so its quotient $\Z_p[G/I_v]^-/(g_v)$ is also $G$-c.t.

Finally, suppose $v$ is inert in $L/L^+$ and $p \mid \# I_v$. 
In this case, the restriction of $j$ to $G/I_v$ is a power of the Frobenius $\varphi_v$, that is, there is a positive integer $m$ such that $j = \varphi_v^m$.
Then
\[
(1 + \# I_v)^m - (\varphi_v^{-1})^m = (1 + \# I_v)^m - j
\]
is a multiple of $g_v$ in $\Z[G/I_v]$.
Since $p \mid \# I_v$ and $j = -1$ in the minus component, the displayed element is a $p$-adic unit in $\Z_p[G/I_v]^-$.
Therefore, $g_v$ is also a unit in $\Z_p[G/I_v]^-$, so we obtain $\Z_p[G/I_v]^-/(g_v) = 0$.
\end{proof}

\begin{defn}\label{defn:SLK}
We define $S(L/K)$ as the set of finite primes of $K$ that are ramified in $L^+/K$ and split in $L/L^+$.
\end{defn}

Now Lemma \ref{lem:A_ct} implies that Theorem \ref{thm:arith_main} can be rephrased as follows:

\begin{cor}\label{cor:arith_main2}
We have
\[
\Cl_L^{T, -} 
\sim \bigoplus_{v \in S(L/K)} \omega^1 (A_v^-).
\]
\end{cor}

\subsection{Realizable classes and admissible classes}\label{ss:formulation}

Let us fix a finite abelian group $\Gamma$ whose order is odd.
We study the category $\cC$ over $\Z'[\Gamma]$.
As explained in the introduction,
we define the set of realizable classes as follows:

\begin{defn}
We say that an element of $\Cmodsim$ is realizable if it is the class of $\Cl_L^{T, -}$ for some extension $L/K$ such that $\Gal(L^+/K) \simeq \Gamma$, where we identify $\Z[\Gal(L/K)]^-$ with $\Z'[\Gamma]$.
The set of realizable classes is denoted by $\cZ^{\rea} \subset \Cmodsim$.
\end{defn}

Next we define the submonoid of admissible classes.
The motivation for the definition will be clear in Corollary \ref{cor:real_adm}.

Let us employ a useful term from group theory.
Given a prime number $p$, a finite group is called {\it $p$-elementary} if it is the product of a $p$-group and a cyclic group.
A finite group is called {\it elementary} if it is $p$-elementary for some prime number $p$.

\begin{defn}\label{defn:Zadm}
\begin{itemize}
\item[(1)]
Associated to $\Gamma$, we define $\wtil{\cS}$ as the set of pairs $(I, \varphi)$, where
\begin{itemize}
\item
$I \subset \Gamma$ is a subgroup,
\item
$I$ is non-trivial,
\item
$I$ is an elementary group, and
\item
$\varphi \in \Gamma/I$ is an element.
\end{itemize}
\item[(2)]
For each $(I, \varphi) \in \wtil{\cS}$,
we define a finite $\Z'[\Gamma]$-module $A_{I, \varphi}$ by
\[
A_{I, \varphi} := \Z'[\Gamma/I]/(1 - \varphi^{-1} + \# I).
\]
\item[(3)]
We define a submonoid $\cZ^{\adm} \subset \Cmodsim$ of admissible classes by
\[
\cZ^{\adm} = \langle [\omega^1(A_{I, \varphi})] \mid (I, \varphi) \in \wtil{\cS} \rangle,
\]
which is generated by $[\omega^1(A_{I, \varphi})]$ for various $(I, \varphi) \in \wtil{\cS}$.
\end{itemize}
\end{defn}

\begin{cor}\label{cor:real_adm}
Let $L/K$ be an extension such that $\Gal(L^+/K) \simeq \Gamma$.
We identify $\Z'[\Gamma]$ and $\Z[\Gal(L/K)]^-$.
Then we have
\begin{equation}\label{eq:Cl_key}
\Cl_L^{T, -} 
\sim \bigoplus_{v \in S(L/K)} \omega^1 (A_{I_v, \varphi_v}).
\end{equation}
In particular, we have $\cZ^{\rea} \subset \cZ^{\adm}$.
\end{cor}

\begin{proof}
By local class field theory, for each prime $v \in S(L/K)$, the inertia group $I_v$ is $p$-elementary for the prime number $p$ lying below $v$.
Therefore, we have $(I_v, \varphi_v) \in \wtil{\cS}$ and also $A_v^- \simeq A_{I_v, \varphi_v}$.
The corollary follows immediately follows from Corollary \ref{cor:arith_main2}.
\end{proof}

Now Theorems \ref{thm:real_p} and \ref{thm:main} are formulated, except for the general definition of $\cT$.

\begin{rem}\label{rem:shift}
The following observation will be used to prove Theorem \ref{thm:main}.
For any integer $n$, the monoid $\cZ^{\adm}$ is isomorphic to the submonoid generated by $[\omega^n(A_{I, \varphi})]$ for various $(I, \varphi) \in \wtil{\cS}$.
This is because the shift automorphisms $\omega^{n-1}$ respect the monoid structure of $\Cmodsim$.
\end{rem}

\begin{rem}\label{rem:plus}
The authors are grateful to Manabu Ozaki, who provided them with the following information.
   
Let $p$ be an odd prime number.
Let $\Gamma$ be any finite $p$-group.
Then, for any finite $\Z_p[\Gamma]$-module $M$, there is a finite abelian extension $L^+/K$ of totally real fields such that $\Gal(L^+/K) \simeq \Gamma$ and $\Z_p \otimes_{\Z} \Cl_{L^+}$ is isomorphic to $M$ as $\Gamma$-modules.

This claim can be shown as follows.
For the given $\Gamma$-module $M$, we consider the semi-direct product of $M \rtimes \Gamma$.
It is a $p$-group, so we may apply the main theorem of Hajir--Maire--Ramakrishna \cite{HMR}.
As a consequence, there exists a totally real field $K$ such that the Galois group of the maximal unramified (not necessarily abelian) $p$-extension over $K$ is isomorphic to $M \rtimes \Gamma$.
Then defining $L^+$ as the intermediate field corresponding to $M$, the requirement is satisfied.
\end{rem}

\section{Concrete applications}\label{sec:4}

In this section, we prove Theorem \ref{thm:main1}.
For this, in \S \ref{ss:lat_clgp}, we compute the lattice associated to the class group explicitly when the Galois group is cyclic.
The general case is doable in principle, but it seems to be complicated.
Then the proof of Theorem \ref{thm:main1} will be given in \S \ref{ss:app_pf}.
In \S \ref{ss:app_ex}, we will also observe numerical examples, which suggest that our theoretical result may be sharp.
A direct generalization of Theorem \ref{thm:main1} will be also provided.

\subsection{Computation of the lattice}\label{ss:lat_clgp}

First we compute the lattice $\Phi(\omega^1(A_{I, \varphi}))$.

\begin{thm} \label{first_main} 
Let $\Gamma$ be a cyclic group of odd order.
Let $(I, \varphi) \in \wtil{\cS}$, which simply means that $I \subset \Gamma$ is a subgroup and $\varphi \in \Gamma/I$ is an element.
Take a lift $\wtil{\varphi} \in \Gamma$ of $\varphi$.
We consider the module $A_{I, \varphi} 
= \Z'[\Gamma/I]/(1 - \varphi^{-1} + \# I)$ over $\Z'[\Gamma]$.
Then we have 
\[
\Phi(\omega^1(A_{I, \varphi})) \sim_{\pe}
(N_I, 1- \wtil{\phi}^{-1}+ \# I),
\]
where we define the norm element $N_I = \sum_{\sigma \in I} \sigma \in \Z[I]$ and the right hand side
 is the ideal of $\Z'[\Gamma]$ generated by the two elements
inside the brackets.
\end{thm}

\begin{proof}
By Lemma \ref{commsq}, we have $\Phi(\omega^1(A_{I, \varphi})) \sim_{\pe} \Omega(\Phi(A_{I, \varphi}))$.
Let $\tau$ be a generator of $I$.
Put $\wtil{g} = 1 - \wtil{\varphi}^{-1} + \# I$.
Then we have $\Phi(A_{I, \varphi}) \sim_{\pe} (\tau - 1, \wtil{g})$, so we have to compute
\[
\Phi(\omega^1(A_{I, \varphi})) \sim_{\pe} \Omega((\tau - 1, \wtil{g})).
\]

Put $R = \Z'[\Gamma]$.
Let $\rho: R^2 \to \mathcal (\tau - 1, \wtil{g})$ be the surjective homomorphism
that sends the first basis element to $\tau-1$ and the second to $\wtil{g}$.
Then by definition $\Omega((\tau - 1, \wtil{g}))$ is projectively equivalent to $\Ker(\rho)$.

We claim that $\Ker(\rho)$ is generated
by $(N_I,0)$ and $(\wtil{g},1-\tau)$. 
Indeed, $(N_I,0), (\wtil{g},1-\tau) \in \Ker(\rho)$ is clear.
Suppose that $(a,b)\in \mathcal \Ker(\rho)$, that is, $a(\tau-1)=-b\wtil{g}$. 
Since $\wtil{g}$ is a non-zero-divisor, $b$ is annihilated by $N_I$, and so we can write $b=b_0(1-\tau)$ for some $b_0 \in R$.
Then $(a,b)-b_0 (\wtil{g},1-\tau)$ is another element in $\Ker(\rho)$ whose second component is zero.
The fact that $(a - b_0 \wtil{g}, 0) \in \Ker(\rho)$ easily
gives that $a - b_0 \wtil{g} \in (N_I)$.
This shows the claim.

It is now easily checked that the first projection $R^2 \to R$ 
gives an isomorphism between $\Ker(\rho)$ and $(N_I, \wtil{g})$.
This completes the proof.
\end{proof}

By Corollary \ref{cor:real_adm} and Theorem \ref{first_main}, we obtain the following:

\begin{thm} \label{firstmain}
In the situation of Corollary \ref{cor:real_adm},
if $\Gamma$ is cyclic, then
we have
\[
\Phi(\Cl_L^{T, -}) \sim_{\pe}
\bigoplus_{v \in S(L/K)}  \bigl( N_{I_v}, 1 - \wtil{\phi_v}^{-1}+ \# I_v \bigr),
\]
where $\wtil{\phi_v}$ is a lift of $\phi_v$.
\end{thm}  \medskip

\subsection{Proof of Theorem \ref{thm:main1}}\label{ss:app_pf}

Now we begin the proof of Theorem \ref{thm:main1}.
As in \S \ref{ss:1-1}, we consider the case where $L^+/K$ is a cyclic $p$-extension, where $p$ is an odd prime number.
Let $F$ be the unique quadratic extension of $K$ in $L$.
Let us write $R = \Z_p[G]^- \simeq \Z_p[\Gamma]$ and 
${}_p \Cl_L^- = \Z_p \otimes_{\Z} \Cl_L^-$.

We begin with the following, which implies that ${}_p \Cl_L^-$ is a cyclic $R$-module in the situation of Theorem \ref{thm:main1}.

\begin{prop}\label{prop:cyc_gen}
Suppose that $L^+/K$ is a cyclic $p$-extension and ${}_p \Cl_F^- = 0$.
Then the $R$-module ${}_p \Cl_L^-$ is generated by $\# S(L/K)$ elements and annihilated by $N_\Gamma$.
\end{prop}

\begin{proof}
The last statement that $N_{\Gamma}$ annihilates ${}_p\Cl_L^-$ 
is a direct consequence of the assumption ${}_p\Cl_F^- =0$.

For the first statement we use genus theory.
By Nakayama's lemma, it is enough to show that the Galois 
coinvariant module $({}_p\Cl_L^-)_{\Gamma}$ is generated by $\# S(L/K)$ elements as a $\Z_p$-module.
Let $H$ be the extension of $L$ that is a subfield of the Hilbert class 
field of $L$ and the Artin map gives an isomorphism $\Gal(H/L) \simeq {}_p\Cl_L^-$.
Then by Galois theory we find an intermediate field $H'$ of $H/L$ such that 
\[
\Gal(H'/L) \simeq ({}_p\Cl_L^-)_{\Gamma}.
\]
Since $\Gamma$ is cyclic, it is known that $\Gal(H'/F)$ is the 
abelianization of $\Gal(H/F)$, that is, $H'$ is the maximal 
abelian extension of $F$ in $H$.

Since $H'/K$ is Galois, the Galois group $\Gal(F/K)$ acts on $\Gal(H'/F)$, so we have a decomposition
\[
\Gal(H'/F) = \Gal(H'/F)^+ \times \Gal(H'/F)^-
\]
with respect to the action of the complex conjugation.
By the construction, we have $\Gal(H'/F)^+ \simeq \Gal(L/F) \simeq \Gamma$ and $\Gal(H'/F)^- \simeq \Gal(H'/L) \simeq ({}_p\Cl_L^-)_{\Gamma}$.

For any finite prime $v$ of $K$, we define a subgroup $I_v \subset \Gal(H'/F)$ as $I_v = \sum_{w \mid v} I_w$, where $w$ denotes the (either one or two) primes of $F$ lying above $v$ and $I_w \subset \Gal(H'/F)$ denotes the inertia group of $w$ in $H'/F$.
Then $I_v$ is stable under the action of $\Gal(F/K)$, so we also have a decomposition $I_v = I_v^+ \times I_v^-$.
Note that $I_v^+$ is identified with the inertia group of $v$ in $\Gal(L^+/K) \simeq \Gamma$.

Since we assume ${}_p \Cl_F^- = 0$, the group $\Gal(H'/F)^-$ is generated by $I_v^-$ for all finite primes $v$ of $K$.
Therefore, the proposition follows if we show that $I_v^- = 0$ unless $v \in S(L/K)$ and, moreover, $I_v^-$ is cyclic when $v \in S(L/K)$.
Since $H'/L$ is unramified, for each prime $w$ of $F$, the inertia group $I_w$ in $\Gal(H'/F)$ is isomorphic to the inertia group of $w$ in $\Gal(L/F) \simeq \Gamma$.
This already shows $I_v^- = 0$ unless $v \in S(L/K)$; if $v$ does not split in $F/K$, then $I_v = I_w \simeq I_v^+$, where $w$ is the unique prime of $F$ lying above $v$.
If $v \in S(L/K)$, there are two primes $w, w'$ of $F$ lying above $v$.
Both $I_w$ and $I_{w'}$ are cyclic since $\Gamma$ is cyclic, and moreover both are isomorphic to $I_v^+$.
Combining this with $I_v = I_w + I_{w'}$, we conclude that $I_v^-$ is cyclic, as claimed.
\end{proof}

From now on, let us assume the hypotheses of Theorem \ref{thm:main1}.
By Proposition \ref{prop:cyc_gen}, we can write ${}_p\Cl_L^- = R/J$ for a suitable ideal $J$.
By Theorem \ref{firstmain}, taking Lemma \ref{lem:no_root} into account, $J$ is projectively equivalent, and even isomorphic,
to $(N_{\Gamma}, p^r)$.
Note that this lattice is non-free, so the case $r = 1$ follows at once.
By Proposition \ref{prop:cyc_gen}, $J$ must contain $N_\Gamma$. 

Therefore, Theorem \ref{thm:main1} follows from the following algebraic proposition:

\begin{prop}\label{prop:ord_alg}
Suppose that $\Gamma$ is a cyclic group of order $p^r$ with a 
prime $p \geq 3$ and $r \geq 2$.
Let $J$ be an ideal of $R = \Z_p[\Gamma]$ such that 
$(N_{\Gamma}) \subset J \subset R$ and $J \simeq (N_{\Gamma}, p^r)$.
Then $\ord_p(\# (R/J))$ is in $\{r, 2r, \cdots, pr\} 
    \cup \{pr + 1, pr + 2, \dots \}$.
\end{prop}

\begin{proof}
By the assumption, there is an element $w \in \Q_p[G]^{\times}$ such that $J = w (N_{\Gamma}, p^r)$.
Then we have
\[
(N_{\Gamma}) \subset w (N_{\Gamma}, p^r) \subset R.
\]

\begin{claim}\label{claim:1}
We have $w \in \frac{1}{p^r}R$ and $\aug(w) \in \Z_p^{\times}$,
where $\aug$ denotes the augmentation.
\end{claim}

\begin{proof}
The claim $w \in \frac{1}{p^r}R$ is clear.
We have $w N_{\Gamma} = \aug(w) N_{\Gamma}$, so $\aug(w) \in \Z_p$ also follows.
It remains to show $\aug(w) \in \Z_p^{\times}$ by using $N_{\Gamma} \in w (N_{\Gamma}, p^r)$.

We have 
\[
w (N_{\Gamma}, p^r) 
= w (N_{\Gamma}, N_{\Gamma} - p^r)
= (\aug(w) N_{\Gamma}, w(N_{\Gamma} - p^r)),
\]
so there are
$x \in \Z_p$ and $y \in R$ such that
\[
N_{\Gamma} = x \aug(w) N_{\Gamma} + y w (N_{\Gamma} - p^r).
\]
Since $N_{\Gamma}(N_{\Gamma} - p^r) = 0$, we have $y (N_{\Gamma} - p^r) = 0$, so this is simplified to
\[
N_{\Gamma} = x \aug(w) N_{\Gamma}.
\]
This says $x \aug(w) = 1$, so the claim follows.
\end{proof}

Now we have
\[
J = w (N_{\Gamma}, p^r) = (N_{\Gamma}, p^r w).
\]
So
\[
R/J \simeq \ol{R}/(\ol{p^r w}),
\]
where we put $\ol{R} = R/(N_{\Gamma})$.

We fix a generator $\sigma$ of $\Gamma$.
We also fix a compatible system $(\zeta_{p^i})$ of $p$-power roots of unity, that is, $\zeta_{p^i}$ is a generator of the group $\mu_{p^i}$ of $p^i$-th roots of unity and we have $(\zeta_{p^i})^p = \zeta_{p^{i-1}}$.
For each $1 \leq i \leq r$, let $\chi_i: \Gamma \to \Z_p[\mu_{p^i}]^{\times}$ 
be the character such that $\chi_i(\sigma) = \zeta_{p^i}$.
We also write $\chi_i$ to mean the induced algebra homomorphism 
$R \to \Z_p[\mu_{p^i}]$.
Then $(\chi_i)_{1 \leq i \leq r}$ gives an injective homomorphism
\[
\ol{R} \hookrightarrow \prod_{i=1}^r \Z_p[\mu_{p^i}].
\]
The cokernel is finite.
Then by a standard argument, we obtain
\[
\# (\ol{R}/(\ol{p^r w}))
= \# \bigg( \prod_{i=1}^r \Z_p[\mu_{p^i}]/(\chi_i(p^r w)) \bigg).
\]
It follows that
\begin{align}
\ord_p(\# (\ol{R}/(\ol{p^r w})))
& = \sum_{i =1}^r \ord_p (\# \Z_p[\mu_{p^i}]/(\chi_i(p^r w)))\\
& = \sum_{i =1}^r \ord_{\Q_p(\mu_{p^i})}(\chi_i(p^r w)),
\end{align}
where $\ord_{\Q_p(\mu_{p^i})}$ denotes the additive valuation on $\Q_p(\mu_{p^i})$, normalized so that we have $\ord_{\Q_p(\mu_{p^i})}(\zeta_{p^i}-1) = 1$.

Here is a quick summary:
Put $c_i := \ord_{\Q_p(\mu_{p^i})}(\chi_i(p^r w))$.
Then we have
\[
\ord_p(\# (R/J)) = \sum_{i = 1}^r c_i.
\]

We have to investigate $c_i$.
By Claim \ref{claim:1}, we have $p^r w \in R$, so there exists 
an element $u \in R$ such that
\[
p^r w - \aug(p^r w) = (\sigma - 1)u.
\]
Then we have
\[
p^r w = (\sigma - 1)u + p^r \aug(w)
\]
and Claim \ref{claim:1} implies $\aug(w) \in \Z_p^{\times}$.

Put $a_i := \ord_{\Q_p(\mu_{p^i})}(\chi_i(u))$.

\begin{claim}\label{claim:2}
If one of $a_1, \dots, a_r$ is less than $p-1$, then we have $a_1 = \dots = a_r$.
\end{claim}

\begin{proof}
For $2 \leq i \leq r$, since $\chi_i(\sigma)^p = \chi_{i-1}(\sigma)$, 
we have $\chi_i(u)^p \equiv \chi_{i-1}(u)$ modulo $(p)$.
Therefore, one of the following holds:
\begin{itemize}
\item
$\ord_p(\chi_i(u)^p) = \ord_p(\chi_{i-1}(u))$, that is, $a_i = a_{i-1}$.
\item
$\ord_p(\chi_i(u)^p) \geq 1$ and $\ord_p(\chi_{i-1}(u)) \geq 1$, 
that is, $a_i \geq p^{i-2}(p-1)$ and $a_{i-1} \geq p^{i-2}(p-1)$.
\end{itemize}
This observation implies the claim (the latter option cannot occur 
for any $i$ by induction).
\end{proof}

Now let us complete the proof of the proposition.
We put
\[
b_i := \ord_{\Q_p(\mu_{p^i})}(\chi_i((\sigma - 1)u)) = 1 + a_i.
\]

\underline{Case 1.}
Suppose one of $a_1, \dots, a_r$ is less than $p-1$.
Then Claim \ref{claim:2} implies 
\[
a_1 = \dots = a_r \in \{0, 1, \dots,  p-2\},
\]
so
\[
b_1 = \dots = b_r \in \{1, 2, \dots,  p-1\}.
\]
Then, since
\[
\ord_{\Q_p(\mu_{p^i})}(p^r \aug(w)) = r p^{i-1}(p-1) \geq r(p-1) > p-1 \geq b_i,
\]
we obtain $c_i = b_i$ for $1 \leq i \leq r$.
Therefore, 
\[
\sum_{i = 1}^r c_i = r b_1 \in \{r, 2r, \cdots, (p-1)r \}.
\]

\underline{Case 2.}
Suppose all of $a_1, \dots, a_r$ are $\geq p -1$.
Then all of $b_1, \dots, b_r$ are $\geq p$.
As in Case 1, we have
\[
\ord_{\Q_p(\mu_{p^i})}(p^r \aug(w)) \geq p,
\]
so we deduce that all of $c_1, \dots, c_r$ are $\geq p$.
Therefore, we have
\[
\sum_{i = 1}^r c_i \geq p r.
\]

This completes the proof of Proposition \ref{prop:ord_alg}.
\end{proof}

This also finishes the proof of Theorem \ref{thm:main1}.

\subsection{Numerical examples}\label{ss:app_ex}

For numerical examples we are forced to choose $K = \Q$, $p=3$, and $r = 2$.
We take the imaginary quadratic field $F$ as one of 
\[
\Q(\sqrt{-1}), \Q(\sqrt{-2}), \Q(\sqrt{-5}), \Q(\sqrt{-6}).
\]
(Note that $\Q(\sqrt{-3})$ is not allowed.)
The class numbers of these are $1, 1, 2, 2$ respectively, so they are prime to $3$.

Also, we take $L^+$ as the unique subfield of $\Q(\mu_q)$ of degree $9$ for some prime $q$ that is congruent to $1$ modulo $9$.
The prime $q$ must split in $F/\Q$, which can be rephrased as a certain congruence condition of $q$ (e.g., when $F = \Q(\sqrt{-1})$, then $q$ is congruent to $1$ modulo $4$).
We consider the primes $q$ in the range $q < 3600$.

Our fields $L$ will be the compositum of $L^+$ and $F$.
The numerical result is that the 3-valuation of the class number of $\Cl_L^-$ takes values
\[
2, 4, 6, 7, 8, 9, 10, 11.
\]
This does not violate the prediction, of course, and also suggests 
that our prediction is sharp.

\begin{rem}
We even did more: even if we remove the condition that $S(L/K)$ consists of a single element (but all $v \in S(L/K)$ are totally ramified in $L^+/K$), a similar reasoning shows that $\ord_p(\# \Cl_L^-)$ is in the set
\[
\{rn, r(n+1), \dots, r (n+p-1) \} \cup \{r(n+p-1)+1, r(n+p-1)+2, \dots\},
\]
where we put $n = \# S(L/K)$.
When $n = 1$, this recovers Theorem \ref{thm:main1}.
We can of course check this generalized prediction for numerical examples.
For instance, for $p = 3$, $r = 2$, and $n = 2$, the possibilities are $4, 6, 8, 9, 10, \dots$. 
This theoretical result can be shown by suitably modifying Proposition \ref{prop:ord_alg}; the details are omitted.
\end{rem}

\section{The realizability problem}  \label{sec:3}

In this section, we prove Theorem \ref{thm:real_p}.
Before that, in \S \ref{ss:real_ex1}, we will illustrate the problem in the simplest non-trivial case, i.e., when $\Gamma$ is the cyclic group whose order is an odd prime number $p$.
The proof of Theorem \ref{thm:real_p} will be given in \S \ref{ss:pf_real}.

\subsection{First case study}  \label{ss:real_ex1}

Let us show the following, which was stated in the introduction:

\begin{thm}\label{thm:p_real}
 Let $\Gamma$ be a cyclic group whose order is an odd prime number $p$ and we work with the coefficient ring $\Z'[\Gamma]$.
Then we have
\[
\cZ^{\rea} = \cZ^{\adm} = \Cmodsim,
\]
that is, every equivalence class of finite $\Gamma$-modules are realized as the class of $\Cl_L^{T, -}$ for some extension $L/K$ with $\Gal(L^+/K) \simeq \Gamma$.
Moreover, we may restrict the base field $K$ to be $\Q$.
\end{thm}

\begin{proof}
Since $\Gamma$ is a $p$-group, the monoid $\Cmodsim$ for $\Z'[\Gamma]$ can be identified with that for $\Z_p[\Gamma]$ (see Proposition \ref{prop:localization}).
Therefore, we may work over $\Z_p[\Gamma]$ instead.
By the interpretation via lattices as in \S \ref{ss:eq_lat}, it is enough to examine
\[
\Phi(\cZ^{\rea}) = \Phi(\Cmodsim)
\]
considered in $\Lat^{\pe}$.
 
In \cite[\S 5.1]{GK}, we showed that $\Cmodsim$ is a free monoid of rank one.
This corresponds to the well-known classification
of $\Z_p[\Gamma]$-lattices that every lattice with constant rank is up to a free
summand a direct sum of copies of $\cM$, where $\cM$ is the maximal order in $\Q_p[\Gamma]$.
Therefore, the basis of $\Phi(\Cmodsim)$ is the class of $\cM$.

On the other hand, by Theorem \ref{firstmain}, $\Phi(\cZ^{\rea})$ consists of the classes of
\[
\bigoplus_{v \in S(L/K)} (N_{\Gamma}, p),
\]
where $L/K$ varies.
Here we used the observation that, for any $v \in S(L/K)$, we have $I_v = \Gamma$ and $\varphi_v$ is trivial since $\Gamma$ is a simple group.
It is easy to see that $(N_{\Gamma}, p) = (N_{\Gamma}, p - N_{\Gamma})$ is isomorphic to $\cM$.

As a result, the theorem follows if we show that for any given integer $n \geq 0$, there is an abelian CM extension $L/\Q$ with $\Gal(L^+/\Q) \simeq \Gamma$ such that $\# S(L/\Q) = n$.
This is a fairly easy exercise.
When $n \geq 1$, take prime numbers $q_1, \dots, q_n$ that are congruent to $1$ modulo $p$, and take $L^+$ as a cyclic extension of $\Q$ of order $p$ in $\Q(\mu_{q_1}, \dots, \mu_{q_n})$ in which all $q_1, \dots, q_n$ are ramified.
By taking an imaginary quadratic field $F$ in which $q_1, \dots, q_n$ are split, we find a desired field as $L = F L^+$.
When $n = 0$, we only have to take $F$ so that the primes are not split.
\end{proof}

\subsection{Proof of Theorem \ref{thm:real_p}}  \label{ss:pf_real}

Now we come back to general $\Gamma$.
By the description in Corollary \ref{cor:real_adm}, we obtain Theorem \ref{thm:real_p} from the following:

\begin{thm}\label{thm:realize}
Let $\Gamma$ be an abstract finite abelian group whose order is odd.
Suppose that we are given a family $(I_1, \varphi_1), \dots, (I_n, \varphi_n) \in \wtil{\cS}$.
Then there exist a totally real field $K$, a finite abelian CM-extension $L/K$, and a group isomorphism $\Gal(L^+/K) \simeq \Gamma$ satisfying the following:
We have $\# S(L/K) = n$ and we can label $S(L/K) = \{v_1, \dots, v_n\}$ so that the inertia group $I_{v_i}$ corresponds to $I_i$ and the Frobenius $\varphi_{v_i}$ in $\Gamma/I_{v_i}$ corresponds to $\varphi_i$.
\end{thm}

To prove this, we make use of the following, which results from global class field theory:

\begin{thm}[{Grunwald--Wang theorem \cite[(9.2.8)]{NSW}}]
Let $K$ be a number field, $G$ a finite abelian group, and $S$ 
a finite set of primes of $K$.
Suppose that for every $v\in S$ we are given a finite abelian extension 
$L_{v}/K_v$ and an embedding $\Gal(L_v/K_v) \hookrightarrow G$.
Suppose that we are not in the special case (in the sense of \cite[(9.1.5), (9.1.7)]{NSW}).
Then there exist a finite abelian extension $L/K$ and an isomorphism 
$\Gal(L/K) \simeq G$ that realizes the designated local extensions 
for $v \in S$.
\end{thm}

\begin{proof}[Proof of Theorem \ref{thm:realize}]
\underline{Step 1.}
First, we construct a totally real field $K$ and distinct 
primes $v_i$ of $K$ ($1 \leq i \leq n$).
The required condition is mild: it is enough to choose them so that there is a surjective homomorphism
\[
\cO_{K_{v_i}}^{\times} \twoheadrightarrow I_i
\]
for each $1 \leq i \leq n$, where $\cO_{K_{v_i}}^{\times}$ denotes the local unit group.
This is possible since, by the definition of $\wtil{\cS}$, for each $1 \leq i \leq n$, there is a prime number $p_i$ such that $I_i$ is $p_i$-elementary.
We may take $v_i$ as a $p_i$-adic prime.

\underline{Step 2.}
We construct a finite abelian extension $L^+_{v_i}/K_{v_i}$ for each $1 \leq i \leq n$.
Let $\wtil{\varphi_i} \in \Gamma$ be a lift of $\varphi_i \in \Gamma/I_i$ and let $D_i \subset \Gamma$ be the subgroup generated by $I_i$ and $\wtil{\varphi_i}$.
Let us choose a uniformizer of $K_{v_i}$, which gives an 
isomorphism $K_{v_i}^{\times} \simeq \cO_{K_{v_i}}^{\times} \times \Z$.
Then, combining the surjective homomorphism $\cO_{K_{v_i}}^{\times} 
\twoheadrightarrow I_i$ in Step 1 with the map $\Z \to D_i$ that sends 
$1$ to $\wtil{\varphi_i}$, we obtain a surjective homomorphism 
$K_{v_i}^{\times} \twoheadrightarrow D_i$.
We define a finite abelian extension $L^+_{v_i}/K_{v_i}$ as the one corresponding to this $K_{v_i}^{\times} \twoheadrightarrow D_i$ via local class field theory.
Then by construction, we have an isomorphism $\Gal(L^+_{v_i}/K_{v_i}) \simeq D_i$ such that the inertia group corresponds to $I_i$ and the Frobenius corresponds to $\varphi_i$.

\underline{Step 3.}
Now we apply the Grunwald--Wang theorem to construct a finite abelian extension $L^+/K$.
We take $S = \{v_1, \dots, v_n\}$ and the local extension for each $v_i$ is $L^+_{v_i}/K_{v_i}$ as in Step 2.
Because the exponent of $\Gamma$ is odd (so not divisible by $4$), we are 
not in ``the special case.''
Therefore, by the Grunwald--Wang theorem, we can construct 
an abelian extension $L^+/K$ and an isomorphism $\Gal(L^+/K) \simeq \Gamma$ such that the localizations at $v_1, \dots, v_n$ are as designated.
Note that this $L^+$ is certainly totally real since the order of $\Gamma$ is odd.

\underline{Step 4.}
We construct a quadratic CM-extension $F/K$ so that the composite field $L = F L^+$ is an extension of $K$ with the desired properties.

Let $S_{\ram}(L^+/K)$ be the set of finite primes of $K$ that are ramified in $L^+$.
We shall construct $F$ satisfying the following:
\begin{itemize}
\item
Each $v_i$ is split in $F/K$ for $1 \leq i \leq n$.
\item
Each $v \in S_{\ram}(L^+/K) \setminus \{v_1, \dots, v_n\}$ does not split in $F/K$.
\end{itemize}
We can find such an $F$ by again using the Grunwald--Wang theorem.
The global Galois group is the cyclic group of order two, so we are not in ``the special case.''
The local extensions for $S_{\ram}(L^+/K)$ are as described above.
The local extensions for archimedean places are all $\mathbb{C}/\R$, so that $F$ is a CM extension of $K$.

Here is a sketch of an alternative construction of $F/K$.
Since it should be a Kummer extension, it is enough to find an element of $K^{\times}$ whose square root generates $F$.
The element should satisfy suitable congruent conditions at primes in $S_{\ram}(L^+/K)$, $2$-adic primes, and archimedean places.
Then the existence follows from the approximation theorem.

Now, by the construction of $F$, if we set $L = F L^+$, we clearly have $S(L/K) = \{v_1, \dots, v_n\}$.
The inertia group and the Frobenius at each $v_i$ are $(I_i, \varphi_i)$ as required, because of the construction of $L^+$ in Steps 2--3 ($F$ does not affect them since $v_i$ is split in $F/K$).
This completes the proof of Theorem \ref{thm:realize}.
\end{proof}

\begin{rem}\label{rem:K_recipe}
In the proof of Theorem \ref{thm:realize}, Step 1 tells us a recipe for the construction of the base field $K$.
If $\Gamma$ is cyclic, then each $I_i$ is also cyclic, so we may take $K = \Q$, thanks to the theorem on arithmetic progressions (cf. Theorem \ref{thm:p_real}).
On the other hand, if $\Gamma$ is not cyclic, we cannot take a uniform $K$ that satisfies Theorem \ref{thm:realize} for all families $\{(I_i, \varphi_i)\}_i$.
\end{rem}

\section{Rephrasing the problem on $\cZ^{\adm}$}\label{sec:6}

In this section, we show Theorem \ref{thm:beta}, which describes the structure of $\cZ^{\adm}$.
It will be a key step to prove Theorem \ref{thm:main}.

\subsection{The key theorem}\label{ss:adm_theorem}

Let $\Gamma$ be a finite abelian group.
In what follows we do not assume that the order of $\Gamma$ is odd and work over $\Z[\Gamma]$ instead of $\Z'[\Gamma]$, which simply widens the scope of the argument.
Let $\Cmodsim$ be the monoid associated to the ring $\Z[\Gamma]$.
As in Definition \ref{defn:Zadm}, we re-define
\[
A_{I, \varphi} := \Z[\Gamma/I]/(1 - \varphi^{-1} + \# I)
\]
(so the former one is recovered by the base-change to $\Z'$ from $\Z$), and then define the submonoid $\cZ^{\adm} \subset \Cmodsim$ in the same way.
We will study the structure of $\cZ^{\adm}$.

\begin{defn}\label{defn:sets}
We define various sets as follows:
\begin{itemize}
\item[(1)]
Let $\cS$ be the set of pairs $(I, D)$, where
\begin{itemize}
\item
	$I \subset D \subset \Gamma$ are subgroups,
   \item
         $I$ is non-trivial,
   \item
	$I$ is an elementary group, and
   \item
         $D/I$ is cyclic.
\end{itemize}
\item[(2)]
For each prime $p$, we write $\Gamma_p$ for the maximal $p$-quotient of $\Gamma$.
Let $\cS_p$ be the set of pairs $(I_p^*, D_p^*)$ 
such that $I_p^* \subset D_p^* \subset \Gamma_p$ are subgroups 
satisfying
\begin{itemize}
\item
$I_p^*$ is non-trivial and 
\item
$D_p^*/I_p^*$ is cyclic.
\end{itemize}
In other words, $\cS_p$ is defined just
as $\cS$, for $\Gamma_p$ instead of $\Gamma$.
Note that $\cS_p = \emptyset$ unless $p \mid \# \Gamma$.
\item[(3)]
Let $\cT$ be the set 
of tuples $(p, H, I_p^*, D_p^*)$ such that
\begin{itemize}
\item
$p$ is a prime number (necessarily a prime divisor of $\# \Gamma$),
\item
$H \subset \Gamma$ is a subgroup such that $\Gamma/H$ is cyclic of order prime to $p$, and
\item
$(I_p^*, D_p^*) \in \cS_p$.
\end{itemize}
\end{itemize}
\end{defn}

\begin{defn}
We define a monoid homomorphism
\[
\beta: \bN^{\cS} \to \bN^{\cT}
\]
by
\[
\beta((I, D)) = \sum_{\substack{D \subset H \\ I_p = I_p^* \\ D_p = D_p^*}} (p, H, I_p^*, D_p^*)
\]
for each $(I, D) \in \cS$, where the sum runs over $(p, H, I_p^*, D_p^*) \in \cT$ satisfying $D \subset H$, $I_p = I_p^*$, and $D_p = D_p^*$.
\end{defn}

Now we can state the key theorem, whose proof will be given in the rest of this section.

\begin{thm}\label{thm:beta}
The monoid $\cZ^{\adm}$ is isomorphic to the image of $\beta: \bN^{\cS} \to \bN^{\cT}$.
\end{thm}

The image of $\beta$ will be studied in \S \ref{sec:7}, which results in Theorem \ref{thm:main}.
For now, let us consider the case where $\Gamma$ is a $p$-group.

\begin{cor}\label{cor:beta_p}
Suppose $\Gamma$ is a $p$-group for some prime number $p$.
Then $\cZ^{\adm}$ is a free monoid of rank $\# \cS = \# \cT$.
\end{cor}

\begin{proof}
By identifying $\Gamma = \Gamma_p$, we have $\cS = \cS_p$.
Moreover, we have $\cS = \cT$ by identifying $(I, D)$ with $(p, \Gamma, I, D)$.
The map $\beta$ is then the identity map.
As a consequence, we obtain the corollary.
\end{proof}

\begin{eg}
Suppose that $\Gamma$ is cyclic of order $p^r$.
Then the choice of $I$ and $D$ is
\[
I = p^i \Gamma, \quad D = p^j \Gamma
\]
with $0 \leq i \leq r-1$ and $0 \leq j \leq i$.
Therefore, we have
\[
\# \cS = \sum_{i = 0}^{r-1} (i+1) = \frac{1}{2} r(r+1).
\]
\end{eg}

\subsection{Reduction to consideration over local rings}\label{ss:red_loc}

For each prime $p$, the ring $\Z_p[\Gamma]$ is decomposed as a product of local rings
\[
\Z_p[\Gamma] \simeq \prod_{\chi} \cO_{\chi}[\Gamma_p],
\]
where $\chi$ runs over a set of representatives
of the characters of $\Gamma$ of order prime to $p$,
modulo $\mathbb Q_p$-conjugacy.
Here, recall that $\Gamma_p$ denotes the maximal $p$-quotient of $\Gamma$.
We will also write $I_p$ and $D_p$ for the maximal $p$-quotient of $I$ and $D$, respectively.
Let $\cC_{p, \chi}$ be the category of finite $\cO_{\chi}[\Gamma_p]$-modules.

For each $(p, \chi)$, as we reviewed in \S \ref{ss:eq_lat},
we have a monoid injective homomorphism
\[
\Phi: (\cC_{p, \chi}) / \modsim \hookrightarrow \Lat_{\cO_{\chi}[\Gamma_p]}^{\pe}.
\]
Moreover, \cite[Theorem 5.2]{GK} implies that $\Lat_{\cO_{\chi}[\Gamma_p]}^{\pe}$ is free on the set of indecomposable $\cO_{\chi}[\Gamma_p]$-lattices that are not projective (i.e., free).
Note that this is true since $\cO_{\chi}[\Gamma_p]$ is a henselian local ring, so the theorem of Krull--Remak--Schmidt--Azumaya holds.

\begin{prop}\label{prop:localization}
The natural monoid homomorphism
\[
\cC / \modsim \to \bigoplus_{p, \chi} (\cC_{p, \chi}) / \modsim
\]
is an isomorphism.
\end{prop}

\begin{proof}
For each $X \in \cC$, since $X$ is finite, $\Z_p \otimes_{\Z} X$ 
is identified with the $p$-Sylow subgroup of $X$ and we have
\[
X \simeq \bigoplus_p (\Z_p \otimes_{\Z} X).
\]
Moreover, any $\Z_p[\Gamma]$-module $Y$ is a direct 
sum of its $\chi$-components $Y_{\chi} := \cO_{\chi}[\Gamma_p] 
\otimes_{\Z_p[\Gamma]} Y$.
In addition, $X$ is $\Gamma$-c.t.~if and only if so are 
all its components.
These observations imply the proposition.
\end{proof}

For a prime number $p$, let us put
\[
{}_p A_{I, \varphi} = \Z_p \otimes_{\Z} A_{I, \varphi}
= \Z_p[\Gamma/I]/(1 - \varphi^{-1} + \# I).
\]
Now we are forced to study the relation among $({}_p A_{I, \varphi})_{\chi}$ (or equivalently among the associated lattices) for various $(I, \varphi) \in \wtil{\cS}$.

\subsection{Reduction to two propositions}

For $(I, \varphi) \in \wtil{\cS}$, define $D \subset \Gamma$ 
to be the subgroup generated by $I$ and a lift of $\varphi$; consequently, 
$\varphi$ generates $D/I$.
The proof of the following two propositions will be given later.

\begin{prop}\label{prop:free}
Let $p$ be a prime number and $\chi$ a character of $\Gamma$ 
whose order is prime to $p$. The following are equivalent:
\begin{itemize}
\item[(i)]
We have $({}_p A_{I, \varphi})_{\chi} \sim 0$.
\item[(ii)]
$I_p$ is trivial or $\chi$ is non-trivial on $D$.
\end{itemize}
\end{prop}

For $(I, \varphi) \in \wtil{\cS}$, let us define $\cL_{I, \varphi} \in \Lat_{\Z[\Gamma]}^{\pe}$ as
the lattice associated to $\omega^{-1}(A_{I, \varphi}) \in \Cmodsim$.
The reason why we consider $\omega^{-1}$ (instead of $\omega^1$) will be explained later.

\begin{prop}\label{prop:equiv_cri}
Let $(I, \varphi)$ and $(I', \varphi')$ be two elements of $\wtil{\cS}$.
Let $p$ be a prime number and $\chi$ a character of $\Gamma$ whose order is prime to $p$.
Suppose that $({}_p A_{I, \varphi})_{\chi} \not\sim 0$ and $({}_p A_{I', \varphi'})_{\chi} \not\sim 0$.
Then the following are equivalent:
\begin{itemize}
\item[(i)]
$({}_p \cL_{I, \varphi})_{\chi} \sim_{\pe} ({}_p \cL_{I', \varphi'})_{\chi}$.
\item[(ii)]
$({}_p \cL_{I, \varphi})_{\chi}$ and $({}_p \cL_{I', \varphi'})_{\chi}$
have a common (nonzero) direct summand in $\Lat_{\cO_{\chi}[\Gamma_p]}^{\pe}$.
\item[(iii)]
We have $I_p = I'_p$ and $D_p = D'_p$.
\end{itemize}
\end{prop}

Let us prove Theorem \ref{thm:beta}, assuming these propositions.

Recall that $\cZ^{\adm}$ is defined as the image of the homomorphism
\[
\bN^{\wtil{\cS}} \to \Cmodsim
\]
that sends $(I, \varphi)$ to $[\omega^1(A_{I, \varphi})]$.
As noted in Remark \ref{rem:shift}, we may consider $\omega^{-1}(A_{I, \varphi})$ instead.

First, for each $(p, \chi)$, let us consider the image 
of $\bN^{\wtil{\cS}} \to \Lat_{\cO_{\chi}[\Gamma_p]}^{\pe}$ 
given by $(I, \varphi) \mapsto ({}_p \cL_{I, \varphi})_{\chi}$.
Thanks to Proposition \ref{prop:free} and Proposition 
\ref{prop:equiv_cri} (i) $\Leftrightarrow$ (ii), the image 
is a free monoid and its basis is the set
\[
\{({}_p \cL_{I, \varphi})_{\chi} \in \Lat_{\cO_{\chi}[\Gamma_p]}^{\pe} \mid (I, \varphi) \in \wtil{\cS}, 
\text{$I_p$ is non-trivial and $\chi$ is trivial on $D$} \}.
\]
Here, projectively equivalent lattices are counted as the same.
Moreover, by Proposition \ref{prop:equiv_cri} (ii) $\Leftrightarrow$ (iii), 
this set is in one-to-one correspondence with the set $\cS_p$ 
by $({}_p \cL_{I, \varphi})_{\chi} \leftrightarrow (I_p, D_p)$.
Consequently, we have a commutative diagram
\[
\xymatrix{
	\N^{\wtil{\cS}} \ar[r] \ar[rd]_{\beta_{p, \chi}}
	& \Lat_{\cO_{\chi}[\Gamma_p]}^{\pe}\\
	& \bN^{\cS_p}, \ar@{^(->}[u]
}
\]
where the map $\beta_{p, \chi}$ sends $(I, \varphi) \in \wtil{\cS}$ to 
\[
\begin{cases}
	(I_p, D_p) \in \cS_p & (\text{if $I_p$ is non-trivial and 
	$\chi$ is trivial on $D$})\\
	0 & (\text{otherwise}).
\end{cases}
\]

Now we vary $p, \chi$.
By the description of $\beta_{p, \chi}$, we obtain the following 
commutative diagram
\[
\xymatrix{
	\N^{\wtil{\cS}} \ar[r]^-{(\beta_{p, \chi})} \ar@{->>}[d]
	& \bigoplus_{(p, \chi)} \bN^{\cS_p} \ar@{^(->}[r]
	& \bigoplus_{(p, \chi)} \Lat_{\cO_{\chi}[\Gamma_p]}^{\pe}\\
	\N^{\cS} \ar[r]_-{(\beta_{p, H})}
	& \bigoplus_{(p, H)} \bN^{\cS_p} \ar@{^(->}[u] &
}
\]
The surjective homomorphism $\bN^{\wtil{\cS}} \to \bN^{\cS}$ is 
induced by the surjective map $\wtil{\cS} \to \cS$ that sends 
$(I, \varphi)$ to $(I, D)$ as before.
The injective homomorphism $\bigoplus_{(p, H)} \bN^{\cS_p} 
\to \bigoplus_{(p, \chi)} \bN^{\cS_p}$ is the diagonal one 
that sends $H$-component to $\chi$-components with $\Ker(\chi) = H$.
Finally, the map $\beta_{p, H}: \bN^{\cS} \to \bN^{\cS_p}$ 
sends $(I, D)$ to
\[
\begin{cases}
	(I_p, D_p) \in \cS_p & (\text{if $I_p$ is non-trivial and $D \subset H$})\\
	0 & (\text{otherwise}).
\end{cases}
\]
Then, identifying $\bigoplus_{(p, H)} \bN^{\cS_p}$ with $\bN^{\cT}$, we may identify the map $(\beta_{p, H})$ as $\beta$.
Thus, we obtain Theorem \ref{thm:beta}, assuming Propositions \ref{prop:free} and \ref{prop:equiv_cri}.

\subsection{Tate cohomology groups}

In this subsection, we deduce Proposition \ref{prop:free} and a part of Proposition \ref{prop:equiv_cri}.
As observed in \cite[Lemma 6.1]{GK}, the definition of $\sim$ implies that equivalent modules in $\cC$ have isomorphic Tate cohomology groups.
So our idea is to compute Tate cohomology groups for various subgroups $H$ of $\Gamma$ (now we consider an arbitrary subgroup $H$ in contrast
with Definition \ref{defn:sets}).

\begin{lem}
For any subgroup $H \subset \Gamma$, 
both $\hat{H}^0(H, A_{I, \varphi})$ and $\hat{H}^{-1}(H, A_{I, \varphi})$ are isomorphic to $\Z[\Gamma/(D + H)]/(\#(I \cap H))$
as $\Z[\Gamma/H]$-modules.
\end{lem}

\begin{proof}
First let us show that
\[
\hat{H}^i(H, \Z[\Gamma/I]) 
\simeq
\begin{cases}
	\Z[\Gamma/(I + H)]/(\#(I \cap H)) & (i = 0)\\
	0 & (i = -1, 1).
\end{cases}
\]
For this, we observe that, for any $i \in \Z$,
\begin{align}
\hat{H}^i(H, \Z[\Gamma/I]) 
& \simeq \Z[\Gamma] \otimes_{\Z[I + H]} \hat{H}^i(H, \Z[(I + H)/I])\\
& \simeq \Z[\Gamma] \otimes_{\Z[I + H]} \hat{H}^i(H, \Z[H/(I \cap H)])\\
& \simeq \Z[\Gamma] \otimes_{\Z[I + H]} \hat{H}^i(I \cap H, \Z)
\end{align}
by using Shapiro's lemma.
When $i = 1$, the claim follows from 
\[
H^1(I \cap H, \Z) = \Hom(I \cap H, \Z) = 0.
\]
To show the claim for $i = -1, 0$, we only have to observe 
that $H_0(I \cap H, \Z) = \Z$, $H^0(I \cap H, \Z) = \Z$, 
and the multiplication by $N_{I \cap H}$ coincides with 
the multiplication by $\# (I \cap H)$ on $\Z$.

By the definition of $A_{i, \varphi}$, we have an exact sequence 
\[
0 \to \Z[\Gamma/I] \overset{1 - \varphi^{-1} + \# I}{\to} 
  \Z[\Gamma/I] \to A_{I, \varphi} \to 0.
\]
Then the lemma follows from the resulting long exact sequence.
Here, we need to use that $\hat{H}^i(H, \Z[\Gamma/I]) $ is 
annihilated by $\# I$.
\end{proof}

\begin{proof}[Proof of Proposition \ref{prop:free}]
Suppose (ii) is false, i.e., $I_p$ is non-trivial and $\chi$ 
is trivial on $D$. Then
\[
\hat{H}^0(I, ({}_p A_{I, \varphi})_{\chi}) 
\simeq \Z_p[\Gamma/D]_{\chi}/(\# I_p)
\]
is nonzero, so (i) is false.

Now suppose (ii) is true.
If $I_p$ is trivial, then $\Z_p[\Gamma/I]$ is $\Gamma$-c.t., 
so ${}_pA_{I, \varphi}$ is also $\Gamma$-c.t.
If $I_p$ is non-trivial and $\chi$ is non-trivial on $D$, 
then $({}_pA_{I, \varphi})_{\chi} = 0$. Therefore, (i) is true.
\end{proof}

\begin{proof}[Proof of Proposition \ref{prop:equiv_cri} 
(i) $\Rightarrow$ (ii) and (i) $\Rightarrow$ (iii)]
(i) $\Rightarrow$ (ii) is clear.
To show (i) $\Rightarrow$ (iii), it is enough to show that 
the module structure of 
\[
\hat{H}^0(H, ({}_p A_{I, \varphi})_{\chi}) 
\simeq \Z_p[\Gamma/(D + H)]_{\chi}/(\# (I \cap H))
\]
allows to recover the groups $I_p$ and $D_p$.
Here, $I_p$ is non-trivial and $\chi$ is trivial on $D$.

For each $p$-subgroup $H$ of $\Gamma$, the order $\# (I \cap H)$ 
is determined by the minimum positive integer that annihilates 
$\hat{H}^0(H, ({}_p A_{I, \varphi})_{\chi})$.
By varying $H$, we thus determine the subgroup $I_p$ of $\Gamma_p$.

Then, by taking the $p$-Sylow subgroup of $I$ as $H$, we know 
the module $\Z_p[\Gamma/D]_{\chi}/(\# I_p)$.
Since $I_p$ is non-trivial, this determines $D_p$.
This is what we wanted.
\end{proof}

We will prove (ii) $\Rightarrow$ (i) and (iii) $\Rightarrow$ (i) 
in the subsequent subsections.

\subsection{The lattice associated to $\omega^{-1}(A_{I, \varphi})$}
  \label{ss:lat_shift_clgp}

To do this, we obtain a concrete description of $\cL_{I, \varphi}$, which was defined as the lattice associated to $\omega^{-1}(A_{I, \varphi})$.
It is a key idea here that $\omega^{-1}(A_{I, \varphi})$ is much easier than $\omega^1(A_{I, \varphi})$, which we described in \S \ref{ss:lat_clgp} only when the group is cyclic.
We write $\nu_I = \sum_{\sigma \in I} \sigma \in \Z[I]$ for the norm element.

\begin{prop}\label{prop:L_explicit}
We have
\[
\cL_{I, \varphi} \sim_{\pe} \Big( \nu_I, 1 - \frac{\nu_I}{\# I} \varphi^{-1} \Big).
\]
\end{prop}

\begin{proof}
We use the computation in \cite[\S 4A]{AK}, which was used to 
determine $\Fitt^{[-1]}(A_{I, \varphi})$.
Let us mention here that the idea here comes from the fact 
that $\Fitt^{[-1]}(A_{I, \varphi})$ is easier than $\Fitt^{[1]}(A_{I, \varphi})$, which 
corresponds to $\Cl_L^{T, -, \vee}$ versus $\Cl_L^{T, -}$.

We have an exact sequence
\[
0 \to \Z[\Gamma/I] \overset{\nu_I}{\to} \Z[\Gamma] \to \Z[\Gamma]/(\nu_I) \to 0.
\]
Let $\wtil{\varphi} \in D$ be a lift of $\varphi \in D/I$.
Put $\wtil{g} = 1 - \wtil{\varphi}^{-1} + \# I \in \Z[\Gamma]$, which 
is of course a lift of $g = 1 - \varphi^{-1} + \# I$.
By the snake lemma, we obtain an exact sequence
\[
0 \to A_{I, \varphi} \to \Z[\Gamma]/(\wtil{g}) \to \Z[\Gamma]/(\wtil{g}, \nu_I) \to 0.
\]
This implies
\[
\omega^{-1}(A_{I, \varphi}) \sim \Z[\Gamma]/(\wtil{g}, \nu_I).
\]
Therefore, by the construction of $\Phi$, we see that $\Phi(\omega^{-1}(A_{I, \varphi}))$ 
is the class of the lattice
\[
(\wtil{g}, \nu_I) \subset \Z[\Gamma].
\]
Let us modify this lattice by multiplying some non-zero-divisors of 
$\Q[\Gamma]$. First we have
\[
\wtil{g}^{-1} (\wtil{g}, \nu_I)
= (1, \nu_I g^{-1}).
\]
Put $h := 1 - \frac{\nu_I}{\# I} \varphi^{-1} + \nu_I$.
Then $\nu_I g = \nu_I h$, so 
\[
h (1, \nu_I g^{-1}) = (h, \nu_I) 
= \Big( \nu_I, 1 - \frac{\nu_I}{\# I} \varphi^{-1} \Big).
\]
This completes the proof.
\end{proof}

From now on, when we write $\cL_{I, \varphi}$, it always means the representative described in this proposition.
For each odd prime number $p$ and a character $\chi$ of $\Gamma$ 
of order prime to $p$, we have
\[
({}_p \cL_{I, \varphi})_{\chi}
= \Big( \nu_{I_p}, 1 - \frac{\nu_{I_p}}{\# I_p} \ol{\varphi}^{-1} \Big)
\]
as lattices of $\cO_{\chi}[G_p]$, where $\ol{\varphi} \in G_p$ 
denotes the image of $\varphi$.

\begin{proof}[Proof of Proposition \ref{prop:equiv_cri} (iii) $\Rightarrow$ (i)]
It is enough to show that (iii) implies that $({}_p \cL_{I, \varphi})_{\chi}$ 
and $({}_p \cL_{I', \varphi'})_{\chi}$ are isomorphic.
Since $\ol{\varphi}$ and $\ol{\varphi'}$ generate the same subgroup 
of $G_p/I_p$, the elements $(1 - \ol{\varphi}^{-1})$ 
and $(1 - \ol{\varphi'}^{-1})$ generate the same ideal of $\Z_p[G_p/I_p]$.
It follows that there is a unit $u \in \Z_p[G_p/I_p]^{\times}$ such that
\[
u (1 - \ol{\varphi}^{-1}) = (1 - \ol{\varphi'}^{-1}).
\]
To ease the notation, let us put $e = \frac{\nu_{I_p}}{\# I_p}$.
Then we have $(eu + (1 - e)) ({}_p \cL_{I, \varphi})_{\chi} 
 = ({}_p \cL_{I', \varphi'})_{\chi}$.
Indeed,
\begin{align}
(eu + (1 - e)) ({}_p \cL_{I, \varphi})_{\chi}
& = \Big( (eu + (1 - e)) \nu_{I_p}, (eu + (1 - e)) (1 - e \ol{\varphi}^{-1}) \Big)\\
& = \Big( eu \nu_{I_p}, eu (1 - \ol{\varphi}^{-1}) + (1 - e) \Big)\\
& = \Big( \nu_{I_p}, 1 - e \ol{\varphi'}^{-1} \Big)\\
& = ({}_p \cL_{I', \varphi'})_{\chi}.
\end{align}
Thus we have proved Proposition \ref{prop:equiv_cri} (iii) $\Rightarrow$ (i).
\end{proof}

\begin{rem}
In fact, we have a more natural proof of Proposition \ref{prop:equiv_cri} (i) $\Leftrightarrow$ (iii).
Let us sketch it.
By Proposition \ref{prop:ext_L} below, the lattice ${}_p \cL_{I, \varphi}$ is an extension of $\Z_p[\Gamma]/(\nu_I)$ by $\Z_p[\Gamma/I]$.
It is possible to directly compute its extension class; we have an isomorphism
\[
\Ext^1_{\Z_p[\Gamma]}(\Z_p[\Gamma]/(\nu_I), \Z_p[\Gamma/I]) \simeq \Z_p[\Gamma/I]/(\# I)
\]
and the extension class corresponds to the class of $\varphi^{-1} - 1$.
Therefore, condition (iii) in Proposition \ref{prop:equiv_cri} claims that the extension classes are the same up to a unit, which indicates that the lattices are isomorphic.
\end{rem}

\subsection{Direct summands of $({}_p \cL_{I, \varphi})_{\chi}$}\label{ss:summand}

Let us study the lattice $\cL_{I, \varphi}$ described in Proposition \ref{prop:L_explicit}, as a preparation
for the missing equivalence of the proof of Proposition \ref{prop:equiv_cri}.

\begin{prop}\label{prop:ext_L}
We have an exact sequence
\begin{equation}\label{eq:EXT}
0 \to \Z[\Gamma/I] \overset{\nu_I}{\to} \cL_{I, \varphi} 
  \to \Z[\Gamma]/(\nu_I) \to 0.
\end{equation}
\end{prop}

\begin{proof}
Consider the natural exact sequence
\[
0 \to \Q[\Gamma/I] \overset{\nu_I}{\to} \Q[\Gamma] 
\overset{\pi}{\to} \Q[\Gamma]/\nu_I \Q[\Gamma] \to 0.
\]
Let us show that this induces the claimed exact sequence, by observing the image and the preimage of $\cL_{I, \varphi}$.

Since $\pi(\nu_I) = 0$ and $\pi \big(1 - \frac{\nu_I}{\# I} 
\varphi^{-1} \big) = 1$, we see that $\pi(\cL_{I, \varphi})$ is generated by $1$ over $\Z[\Gamma]$.
The natural homomorphism $\Z[\Gamma]/(\nu_I) \to 
\Q[\Gamma]/\nu_I \Q[\Gamma]$ is injective 
and its image is generated by $1$ over $\Z[\Gamma]$.
Therefore, the image of $\cL_{I, \varphi}$ is $\Z[\Gamma]/(\nu_I)$, as claimed.

To determine the preimage,
let $a \in \Q[\Gamma/I]$ be any element such that $\nu_I a \in \cL_{I, \varphi}$.
We want to show $a \in \Z[\Gamma/I]$.
Let us take elements $b \in \Z[\Gamma/I]$ and $c \in \Z[\Gamma]$ 
such that $\nu_I a = \nu_I b + \big(1 - \frac{\nu_I}{\# I} \varphi^{-1} \big) c$.
Then $\nu_I(a - b) = \big(1 - \frac{\nu_I}{\# I} \varphi^{-1} \big) c$.
In particular, this equation implies $c \in \nu_I \Q[\Gamma]$, so
\[
c \in \Z[\Gamma] \cap \nu_I \Q[\Gamma] = \nu_I \Z[\Gamma] = (\nu_I).
\]
Then
\[
\nu_I (a - b) = \Big(1 - \frac{\nu_I}{\# I} \varphi^{-1} \Big) c
= (1 - \varphi^{-1} ) c
\in (\nu_I).
\]
This implies $\nu_I a \in (\nu_I)$, so $a \in \Z[\Gamma/I]$, as claimed.
This completes the proof.
\end{proof}

\begin{prop}\label{prop:indec}
Let $p$ be a prime number and $\chi$ a character of $\Gamma$ 
of order prime to $p$.
Then
either $({}_p \cL_{I, \varphi})_{\chi}$ is indecomposable 
over $\cO_{\chi}[G_p]$,  or the sequence
\[
0 \to \Z_p[G/I]_{\chi} \to ({}_p \cL_{I, \varphi})_{\chi} 
  \to (\Z_p[G]/(\nu_I))_{\chi} \to 0,
\]
which is obtained by Proposition \ref{prop:ext_L},
splits.
\end{prop}

\begin{proof}
Note that both $\Z_p[G/I]_{\chi}$ and $(\Z_p[G]/(\nu_I))_{\chi}$ 
are indecomposable unless zero, since they are cyclic modules 
over a local ring.

Suppose that there is a decomposition $({}_p \cL_{I, \varphi})_{\chi} 
= M_1 \oplus M_2$ with nonzero $M_1$ and $M_2$.
Since $(\Z_p[G]/(\nu_I))_{\chi}$ is a cyclic module, by Nakayama's lemma, 
we may assume that the map $M_1 \to (\Z_p[G]/(\nu_I))_{\chi}$ is surjective.

We claim that the map $M_2 \to (\Z_p[G]/(\nu_I))_{\chi}$ is zero.
For this, we may work after base-change from $\Z_p$ to $\Q_p$ 
so that everything is semi-simple.
Then $M_1$ and $M_2$ have no common irreducible components 
as $M_1 \oplus M_2$ is (generically) free of rank one.
Since there is a surjective map from $M_1$ to $(\Z_p[G]/(\nu_I))_{\chi}$, 
we see that $(\Z_p[G]/(\nu_I))_{\chi}$ and $M_2$ have 
no common irreducible components. This shows the claim.

Now by the displayed exact sequence, $\Z_p[G/I]_{\chi}$ is isomorphic 
to $\Ker(M_1 \to (\Z_p[G]/(\nu_I))_{\chi}) \oplus M_2$.
Therefore, $M_1 \to (\Z_p[G]/(\nu_I))_{\chi})$ is isomorphic, 
so the sequence splits.
\end{proof}

\begin{proof}[Proof of Proposition \ref{prop:equiv_cri} (ii) $\Rightarrow$ (i)]
Suppose that $({}_p \cL_{I, \varphi})_{\chi}$ and 
$({}_p \cL_{I', \varphi'})_{\chi}$ have a common direct summand.
We want to show that then these
two lattices are indeed isomorphic.
If one of them is indecomposable, then the claim is clear
(notice that the $\cO_{\chi}$-ranks of $({}_p \cL_{I, \varphi})_{\chi}$ and 
$({}_p \cL_{I', \varphi'})_{\chi}$ are the same).
Suppose that both are decomposable.
By Proposition \ref{prop:indec}, we have
\[
({}_p \cL_{I, \varphi})_{\chi} \simeq \Z_p[G/I]_{\chi} 
  \oplus (\Z_p[G]/(\nu_I))_{\chi} 
\]
and similarly for $({}_p \cL_{I', \varphi'})_{\chi}$.
The assumption implies that one of the following holds:
\[
\begin{cases}
\Z_p[G/I]_{\chi} \simeq \Z_p[G/I']_{\chi}\\
\Z_p[G/I]_{\chi} \simeq (\Z_p[G]/(\nu_{I'}))_{\chi} \\
(\Z_p[G]/(\nu_I))_{\chi} \simeq \Z_p[G/I']_{\chi}\\
(\Z_p[G]/(\nu_I))_{\chi} \simeq (\Z_p[G]/(\nu_{I'}))_{\chi} 
\end{cases}
\]
Neither the second nor the third isomorphism can hold, because
one side contains the trivial character component and the other does not.
Therefore, the first or the fourth occurs, which implies $I_p = I_p'$ and the desired isomorphism of lattices follows.
This completes the proof.
\end{proof}

\section{The structure of $\cZ^{\adm}$}\label{sec:7}

In this section, we prove Theorem \ref{thm:main} by using Theorem \ref{thm:beta}.
Let $\Gamma$ be a finite abelian group.
The case where $\Gamma$ is a $p$-group was done in Corollary \ref{cor:beta_p}.
The case where $\Gamma$ is cyclic will be done in \S \ref{ss:cyc}, and the other cases will be in \S \ref{ss:non-cyc}.

\subsection{Useful observations}

To study the image of $\beta$, the following is useful.

\begin{lem}\label{lem:I_dec}
Let $(I, D) \in \cS$ and we suppose $D$ is cyclic.
Then we have
\[
\beta((I, D)) = \sum_{p \mid \# I} \beta((I_{(p)}, D)),
\]
where $I_{(p)}$ denotes the $p$-Sylow subgroup of $I$.
Here we have $(I_{(p)}, D) \in \cS$ thanks to the assumption that $D$ is cyclic.
\end{lem}

\begin{proof}
This can be checked directly from the definition of $\beta$.
\end{proof}

\begin{cor}\label{cor:S'}
Define a subset $\cS' \subset \cS$ by
\[
\cS' = \{ (I, D) \in \cS \mid \text{either $D$ is non-cyclic 
  or $\# I$ is a prime-power}\}.
\]
Then we have $\beta(\bN^{\cS'}) = \beta(\bN^{\cS})$.
\end{cor}

\begin{proof}
For each $(I, D) \in \cS \setminus \cS'$, we have 
that $D$ is cyclic (and $\# I$ is not a prime-power, but formally this property is unnecessary for now).
For each $p \mid \# I$, defining $I_{(p)}$ as in Lemma \ref{lem:I_dec}, 
we have $(I_{(p)}, D) \in \cS'$ since $\# I_{(p)}$ is a prime-power.
Then Lemma \ref{lem:I_dec} implies that 
$\beta((I, D)) \in \beta(\bN^{\cS'})$.
\end{proof}

According to this corollary, we only have to study the image 
of the homomorphism
\[
\beta' = \beta|_{\bN^{\cS'}}: \bN^{\cS'} \to \bN^{\cT}.
\]
Let us compare the cardinalities of $\cS'$ and $\cT$.

\begin{prop}\label{prop:card_comp}
We have $\# \cS' \geq \# \cT$ and the equality holds if and only 
if  $\Gamma$ is cyclic or $\# \Gamma$ is a prime-power.
\end{prop}

\begin{proof}
We define
\begin{align}
\cS'' & = \{ (I, D) \in \cS \mid \text{$\# I$ is a prime-power}\}\\
& = \coprod_{p \mid \# \Gamma} \{ (I, D) \in \cS \mid \text{$\# I$ is a $p$-power}\}.
\end{align}
Then it is clear that $\cS' \supset \cS''$.
Moreover, it is easy to see that $\# \cS'' = \# \cT$.
Therefore, we have $\# \cS' \geq \# \cT$.
The equality is equivalent to $\cS'' = \cS'$.
The equality fails if and only if there is $(I, D)$ such that $D$ is non-cyclic 
and $\# I$ is non-prime-power.
Such a pair $(I, D)$ exists if and only if $\Gamma$ is non-cyclic and 
$\# \Gamma$ is non-prime-power.
\end{proof}

\begin{rem}
The authors conjecture that the homomorphism $\beta'' = \beta|_{\bN^{\cS''}}: \bN^{\cS''} \to \bN^{\cT}$ is injective.
This is true when $\Gamma$ is cyclic or $\# \Gamma$ is a prime-power, i.e., when $\cS' = \cS''$ (see Corollary \ref{cor:beta_p} and Proposition \ref{prop:cyc_inj}).
However, we have not proved this for general $\Gamma$.
\end{rem}

\subsection{The case of cyclic groups}\label{ss:cyc}

In this subsection, we prove Theorem \ref{thm:main}(1) for cyclic groups $\Gamma$.
Thanks to Corollary \ref{cor:S'}, it is enough to show the following:

\begin{prop}\label{prop:cyc_inj}
When $\Gamma$ is cyclic, the homomorphism $\beta': \bN^{\cS'} 
\to \bN^{\cT}$ is injective.
\end{prop}

\begin{proof}
As in Proposition \ref{prop:card_comp}, we have
\[
\cS' = \cS'' 
= \coprod_{p \mid \# \Gamma} \{ (I, D) \in \cS \mid \text{$\# I$ is a $p$-power}\}.
\]
By definition, $\cT$ is also decomposed as a disjoint union
\[
\cT = \coprod_p \big( \{ H \subset \Gamma \} \times \cS_p \big),
\]
where $\{ H \subset \Gamma\}$ denotes the set of subgroups such that $\Gamma/H$ is cycic of order prime to $p$.
Also, the homomorphism $\beta$ respects these decompositions.
Therefore, it is enough to check the injectivity for each components.
The proposition follows from the next lemma, applied for the prime-to-$p$-component of $\Gamma$ as $\Delta$.
\end{proof}

\begin{lem}
For a finite abelian group $\Delta$, let
\[
\bigoplus_{D \subset \Delta} \bN \to \bigoplus_{H \subset \Delta} \bN,
\]
where $D$ and $H$ run over all subgroups of $\Delta$, be the homomorphism defined by
\[
\sum_D a_D [D] \mapsto \sum_H \bigg(\sum_{D \subset H} a_D\bigg) [H].
\]
Then this homomorphism is injective.
\end{lem}
\begin{proof}
Indeed, we can recover $a_D$ from the family of values 
$\big(\sum_{D \subset H} a_D\big)_H$ by induction on the size of $D$.
\end{proof}

Note that in this lemma, it is important that $H$ runs over all subgroups.
However, in the definition of $\cT$, the quotient group $\Gamma/H$ must be cyclic.
So we need to use the assumption that $\Gamma$ is cyclic again.

Before closing this subsection, it is worth mentioning the cardinality of $\cS$ and $\cT$ when $\Gamma$ is a cyclic group.

\begin{lem}
Suppose
\[
\# \Gamma = p_1^{e_1} \cdots p_s^{e_s},
\]
where $p_1, \dots, p_s$ are distinct primes and $e_i \geq 1$.
Then we have
\[
\# \cS = \prod_{i = 1}^s \frac{1}{2} (e_i + 1)(e_i + 2) 
    - \prod_{i = 1}^s (e_i + 1)
\]
and
\[
\# \cT = \frac{1}{2} \bigg(\sum_{i = 1}^s e_i \bigg) 
   \prod_{i = 1}^s (e_i + 1).
\]
\end{lem}

\begin{eg}
If $e_1 = \cdots = e_s = 1$, then $\# \cS  = 3^s - 2^s$ and 
$\# \cT = s \cdot 2^{s-1}$.
\end{eg}

\begin{proof}
Since $\Gamma$ is cyclic, the set $\cS$ consists of pairs $(I, D)$ 
such that $0 \neq I \subset D \subset \Gamma$.
This corresponds to divisors $1 \neq \# I \mid \# D \mid \# \Gamma$.
The number of such pairs is
\[
\prod_{i = 1}^s \# \{ 0 \leq a \leq b \leq e_i \} 
- \prod_{i = 1}^s \# \{ 0 \leq b \leq e_i \},
\]
which is equal to the claimed formula.

Next we consider $\cT$.
For each $1 \leq i \leq s$, the number of subgroups $H \subset \Gamma$ such that $\Gamma/H$ is cyclic of order prime to $p_i$ is equal to
\[
\prod_{j \neq i} \# \{ 0 \leq c \leq e_j \}
= \prod_{j \neq i} (e_j + 1).
\]
Therefore, we obtain
\[
\# \cT = \sum_{i = 1}^s \bigg( \prod_{j \neq i} (e_j + 1) \bigg) \cdot \# \cS_{p_i}.
\]
We also have
\[
\# \cS_{p_i} = \# \{ 1 \leq a \leq b \leq e_i \}
= \frac{1}{2} e_i (e_i + 1).
\]
By combining these, we obtain the lemma.
\end{proof}

\subsection{The case of non-cyclic groups}\label{ss:non-cyc}

In this subsection we prove Theorem \ref{thm:main}(2).
To show a monoid is non-free, 
we will count the irreducible elements of the monoid as in \cite[\S 5.2]{GK}.

Recall that an irreducible element of a commutative monoid 
is a non-invertible element that cannot be represented as a 
sum of two non-invertible elements.
Then the basis of a free commutative monoid is determined as 
the set of irreducible elements.

\begin{prop}\label{prop:irred}
For $(I, D) \in \cS$, we have $\beta((I, D))$ is an irreducible 
element of $\beta(\bN^{\cS})$ if and only if $(I, D) \in \cS'$.
Also, $\beta$ is injective on $\cS'$.
\end{prop}

\begin{proof}
The ``only if'' part follows from the proof of Corollary \ref{cor:S'}.
Let us show the ``if'' part.
Let $(I, D) \in \cS'$ and suppose that
\[
\beta((I, D)) = \sum_{\lambda} \beta((I_{\lambda}, D_{\lambda}))
\]
for a family $\{(I_{\lambda}, D_{\lambda})\}_{\lambda \in \Lambda}$ in $\cS$.
We want to show $\Lambda$ is a singleton and the family 
coincides with $\{(I, D)\}$.

For each prime $p \mid \# \Gamma$, we have
\[
\sum_{H \supset D} (p, H, I_p, D_p)
  = \sum_{\lambda \in \Lambda, p \mid \# I_{\lambda}} 
  \sum_{H \supset D_{\lambda}} (p, H, (I_{\lambda})_p, (D_{\lambda})_p).
\]
In the both sides, $H$ satisfies $\Gamma/H$ is cyclic and $p \nmid [\Gamma: H]$.
Also, the left side should be understood to be zero unless $p \mid \# I$.
This equality can be rephrased as the combination of
\begin{itemize}
\item[(a)]
For any $p \mid \# I$ and $\lambda \in \Lambda$, we have either 
$p \nmid \# I_{\lambda}$ or $(I_p, D_p) = ((I_{\lambda})_p, (D_{\lambda})_p)$.
\item[(b)]
We have
\[
\sum_{H \supset D} (p, H, I_p, D_p)
= \sum_{\lambda \in \Lambda, p \mid \# I_{\lambda}} \sum_{H \supset D_{\lambda}} (p, H, I_p, D_p).
\]
\end{itemize}
By considering the $H = \Gamma$ component, we can divide (b) as
\begin{itemize}
\item[(b1)]
We have $\prim(\# I) = \coprod_{\lambda \in \Lambda} 
\prim(\# I_{\lambda})$, where $\prim(n)$ denotes the set of 
prime divisors of $n$.
In other words, we have $\prim(I_{\lambda}) \subset \prim(I)$ 
for any $\lambda \in \Lambda$ and moreover, for each $p \mid \# I$, 
there exists a unique $\lambda_p \in \Lambda$ such that 
$p \mid \# I_{\lambda_p}$.
\item[(b2)]
For that $\lambda_p$, we have
\[
\sum_{H \supset D} (p, H, I_p, D_p)
= \sum_{H \supset D_{\lambda_p}} (p, H, I_p, D_p).
\]
\end{itemize}

We claim that (b2) is equivalent to $D_{\lambda_p} = D$ (assuming (a)).
To show this, we use the following:

\begin{lem}
Let $\Gamma$ be a finite abelian group.
Let $D, D'$ be two subgroups of $\Gamma$.
Suppose that for any subgroup $H \subset \Gamma$ such that 
$\Gamma/H$ is cyclic, we have $H \supset D$ if and only if $H \supset D'$.
Then we have $D = D'$.
\end{lem}

\begin{proof}
This is because any subgroup of $\Gamma$ can be expressed as the intersection of subgroups $H \subset \Gamma$ with $\Gamma/H$ is cyclic.
\end{proof}

Since we impose an additional condition $p \nmid [\Gamma: H]$, 
we deduce from (b2) only that the prime-to-$p$ component of $D$ 
coincides with that of $D_{\lambda_p}$.
But by combining this with (a), we obtain $D_{\lambda_p} = D$, 
as claimed.

Note that, since any $\lambda$ is represented as $\lambda_p$ for some $p$, 
we obtain $D_{\lambda} = D$ for any $\lambda \in \Lambda$.
As a summary, we have observed:
\begin{itemize}
\item[(b1)]
We have $\prim(\# I) = \coprod_{\lambda \in \Lambda} \prim(\# I_{\lambda})$.
\item[(a)']
For each $p$ and the $\lambda_p \in \Lambda$, we have $(I_{\lambda_p})_p = I_p$.
\item[(b2)']
For any $\lambda \in \Lambda$, we have $D_{\lambda} = D$.
\end{itemize}

Now we use the assumption that $(I, D) \in \cS'$, that is, 
either $D$ is non-cyclic or $\# I$ is a prime-power.
If $\# I$ is a $p$-power, then (b1) implies that $\Lambda$ 
is a singleton and then (a)' and (b2)' imply 
$(I_{\lambda_p}, D_{\lambda_p}) = (I, D)$, as desired.

Suppose $D$ is non-cyclic.
Then there is $p \mid \# D$ such that the $p$-group $D_p$ 
is non-cyclic.
Then (b2)' implies that $(D_{\lambda})_p$ is also non-cyclic 
for any $\lambda$.
Since $D_{\lambda}/I_{\lambda}$ is cyclic, this implies 
that $(I_{\lambda})_p$ is non-zero for any $\lambda$.
By (b1), we must have $\Lambda$ is a singleton.
By (b1), (a)', and (b2)', we have $(I_{\lambda_p}, D_{\lambda_p}) 
= (I, D)$, as desired.

This completes the proof of Proposition \ref{prop:irred}.
\end{proof}

Now we show Theorem \ref{thm:main}(2).
Suppose $\beta(\bN^{\cS})$ is a free monoid.
As a summary of Corollary \ref{cor:S'} and Proposition \ref{prop:irred}, we have a surjective homomorphism
\[
\beta': \bN^{\cS'} \to \beta(\bN^{\cS}),
\]
which yields a bijection from $\cS'$ to the set of irreducible elements of $\beta(\bN^{\cS})$.
Therefore, the rank of $\beta(\bN^{\cS})$ is equal to $\# \cS'$.
On the other hand, since $\beta(\bN^{\cS})$ is a submonoid of $\bN^{\cT}$, 
its rank must be $\leq \# \cT$.
As a consequence, we must have $\# \cS' \leq \# \cT$.
By Proposition \ref{prop:card_comp}, we deduce that 
either $\Gamma$ is cyclic or $\# \Gamma$ is a prime-power.
This completes the proof of Theorem \ref{thm:main}(2).

\section*{Acknowledgments}

We sincerely thank the anonymous referees for providing valuable comments to improve this paper.
The second author is supported by JSPS KAKENHI Grant Number 22K13898.

\bibliographystyle{abbrv}
\bibliography{mybib}

\end{document}